\documentclass[reqno,11pt]{amsart}
\usepackage{amsmath,amssymb,amsfonts,amsthm, amscd,indentfirst}
\usepackage{amsmath,latexsym,amssymb,amsmath,
	amscd,amsthm,amsxtra,mathrsfs}
 
\usepackage[top=3cm, bottom=3.5cm, right=2.5cm, left=2.5cm]{geometry}

\usepackage[hyphens]{url}
\usepackage{hyperref}
\usepackage{bookmark}
\usepackage{amsmath,thmtools,mathtools}
\allowdisplaybreaks 
\newcommand{\CP}{\mathbb{CP}}

\mathtoolsset{showonlyrefs=true}
\makeatletter
\def\@dotsep{10000} 
\makeatother

\newcommand{\C}{\mathbb{C}}

\usepackage[msc-links]{amsrefs} 

\newcounter{mparcnt}

\usepackage{fancyhdr}
\usepackage{esint}
\usepackage{enumerate}
\usepackage{xcolor}

\usepackage{pictexwd,dcpic}
\usepackage{graphicx}
\usepackage{caption}

\usepackage{etoolbox}

\usepackage{slashed}

\usepackage{epsfig,here}
\usepackage{subfigure,here}

\theoremstyle{plain}
\newtheorem{theorem}{Theorem}[section]
\newtheorem{lemma}[theorem]{Lemma}
\newtheorem{proposition}[theorem]{Proposition}
\newtheorem{corollary}[theorem]{Corollary}

\theoremstyle{definition}

\theoremstyle{remark}
\newtheorem{remark}[theorem]{Remark}

\newcommand{\abs}[1]{\lvert#1\rvert}

\newcommand{\norm}[1]{\lVert#1\rVert}

\newcommand{\rd}{{\rm d}}

\newcommand{\rdV}{{\rm dV}}

\newcommand{\rid}{{\rm id}}
\newcommand{\rtr}{{\rm tr}}

\newcommand{\ip}{\lrcorner\,}
\newcommand{\w}{\wedge}
\newcommand{\p}{\partial}

\def\<{\langle}
\def\>{\rangle}
\def\S{\mathbb{S}}
\def\R{\mathbb{R}}
\def\C{\mathbb{C}}

\newcommand{\eq}[1]{\begin{equation}\allowdisplaybreaks\begin{alignedat}{2} #1 \end{alignedat}\end{equation}}

\numberwithin{equation} {section}

\begin{document}

	
\title[Harmonic maps from $\S^{2n+1}$ into $\mathbb{CP}^n$ with least Morse index]{
Harmonic maps from $\S^{2n+1}$ into $\mathbb{CP}^n$ with least Morse index
}

\date{\today}


\author{Qun Chen}
\address{ Wuhan University, School of Mathematics and Statistics, 430072 Wuhan, China}
\email{qunchen@whu.edu.cn}

\author{Guofang Wang}
\address{Albert-Ludwigs-Universit\"at Freiburg,
Mathematisches Institut,
Ernst-Zermelo-Str. 1,
D-79104 Freiburg, Germany}
\email{guofang.wang@math.uni-freiburg.de}

\author{Mingwei Zhang}
\address{ Wuhan University, School of Mathematics and Statistics, 430072 Wuhan, China and 
Albert-Ludwigs-Universit\"at Freiburg,
Mathematisches Institut,
Ernst-Zermelo-Str. 1,
D-79104 Freiburg, Germany}
\email{zhangmwmath@whu.edu.cn}

\begin{abstract} 

Let $u:\S^{2n+1}\to\mathbb{CP}^n$ be a smooth nonconstant harmonic map. 
We prove that its Morse index is equal to $ 2n+2$ if and only if
\eq{
    u = h \circ \pi \circ \Xi,
}
where $\Xi:\S^{2n+1}\to\S^{2n+1}$ is an isometric transformation, $\pi:\S^{2n+1}\to\mathbb{CP}^n$ is the Hopf map and $h:\mathbb{CP}^n\to\mathbb{CP}^n$ is a holomorphic map with degree one. When $n=1$, it 
refines  a classical result of Urakawa  \cite{Urakawa87}    and 
a recent result of Rivi\`{e}re \cite{Riviere23}. Moreover, we  prove that $h\circ\pi\circ \Xi$ has nullity $3n^2+5n$.


\end{abstract}



\maketitle
\setcounter{footnote}{0}


\tableofcontents
\setcounter{footnote}{0}

\section{Introduction}

The Morse index and stability are central concepts  in calculus of variations and differential geometry.
In this paper, we study the Morse index and nullity of the Hopf map (or Hopf fibration) $\pi:\S^{2n+1}\to\mathbb{CP}^n$, which is a classical example of a harmonic map.

In the case $n=1$, i.e., the case $\S^3\to\S^2$,  Urakawa  \cite{Urakawa87} proved a classical result.

\medskip 

\noindent{\bf Theorem A} (Urakawa  \cite{Urakawa87}).
 {\it The Hopf map $\pi: \S^3\to\S^2$ has Morse index $4$ and nullity $8$.}

\medskip

Recently Rivi\`{e}re \cite{Riviere23} proved a  rigidity result, which is analogous to Urbano's index characterization for minimal surfaces in $\S^3$ \cite{Urbano90}.
\medskip

\noindent{\bf Theorem B} (Rivi\`{e}re \cite{Riviere23})
{\it Let $u:\S^3\to\S^2$ be a smooth non-constant harmonic map with ${\rm ind}(u)= 4$. Then $u$ factors, up to an isometry of $\S^3$, as the Hopf map $\pi$ followed by a holomorphic self-map $h$ of $\S^2$.}

The main motivation of his work is to use the ideas in 
the proof of
the Willmore Conjecture by  Marques--Neves \cite{MarquesNevesWillmore2014} to show the following conjecture:

\smallskip

\noindent {\bf Conjecture} (Rivi\`{e}re \cite{Riviere98}) \textit{The minimizers of the 3-energy among homotopically non zero maps from $\S^3$ into $\S^2$ are exactly given by the composition of the Hopf maps $\pi$ with conformal transformations and isometries of $\S^3$.}

\smallskip 

This conjecture shares many similarities with the Willmore conjecture. Recall that in the proof of the latter one, the index characterization of F. Urbano \cite{Urbano90} plays a key role: {\it any compact orientable non-totally geodesic minimal surface with Morse index $5$ in $\S^3$ must be the Clifford torus}. 
The conjecture of   Rivi\`{e}re  was proved in the recent work of the second and the third authors in \cite{WZ26curl} with a complete different method. Nevertheless, the approach proposed by him  has a significant interest in geometric analysis. 

Therefore, we first  refine Theorem B by showing that the holomorphic map $h$ must have degree one. Then it arises the following question: 

\smallskip

{\it For any holomorphic self-map $h$ of $\S^2$ of degree one, does its composition $h\circ \pi$ with the Hopf map $\pi$ have Morse index $4$?} 

\smallskip

To the best of our knowledge, it has been open until the present paper. This question was addressed in \cite{Kaltefleiter23} under a more restrictive assumption. We will give this question an affirmative answer. Finally we generalize these two results to higher dimensions.

We now state our main result.

\begin{theorem}\label{main_thm}
Let $u:\S^{2n+1}\to\mathbb{CP}^n$ be a smooth non-constant harmonic map.  The Morse index ${\rm ind}(u)=2n+2$ if and only if 
\eq{\label{eq0}
    u = h \circ \pi \circ \Xi,
}
where $\Xi:\S^{2n+1}\to\S^{2n+1}$ is an isometric transformation, $\pi:\S^{2n+1}\to\mathbb{CP}^n$ is the Hopf map and $h:\mathbb{CP}^n\to\mathbb{CP}^n$ is a holomorphic map with degree one.
\end{theorem}

There have been a lot of results on stability and Morse index of harmonic maps. Here we just mention closely related ones. 
For harmonic maps from $\S^m$, Xin \cite{Xin80} proved that every non-constant harmonic map with $m\ge 3$ is unstable, and El Soufi \cite{ElSoufi95} strengthened it by showing that its index is at least $m+1$ (see also \cite{Ni23}). In contrast, for maps from $\S^2$ there exist many stable examples; see Siu--Yau \cite{Siu-Yau80}.
For harmonic maps into $\mathbb{CP}^n$, Chen \cite{Chen96} showed that the stability of a harmonic map implies
\eq{\label{eq_qhwc}
    [ (\rd u) (\rd u)^*,J]=0,
}
which Loubeau \cite{Loubeau97} called \emph{pseudo horizontal weak conformality}. Here $J$ denotes the standard complex structure on $\mathbb{CP}^n$. When $n=1$ (so $\mathbb{CP}^1=\S^2$), \eqref{eq_qhwc} is equivalent to horizontal weak conformality; in particular, such a harmonic map is a harmonic morphism (see Fuglede \cite{Fuglede78} or Ishihara \cite{Ishihara79}), and factors as described above (see Baird--Wood \cite{Baird-Wood88}). For further discussions we refer to \cite{Baird-Wood03}. Related rigidity results for stable harmonic maps into compact Hermitian symmetric spaces appear in \cites{Burns-deB88, BBdBR89, Urakawa87}.

From  now on, we focus on maps from $\S^{2n+1}$ to $\CP^n$. 
In order to understand well our proof's strategy for Theorem \ref{main_thm}, we divide it into the following three Theorems.

First, we prove a generalization of Theorem B.

\begin{theorem}\label{Theorem_R}
Let $u:\S^{2n+1}\to\mathbb{CP}^n$ be a smooth non-constant harmonic map.
 If the Morse index ${\rm ind}(u)=2n+2$, then
\eq{
    u = h \circ \pi \circ \Xi,
}
where $\Xi:\S^{2n+1}\to\S^{2n+1}$ is an isometric transformation, $\pi:\S^{2n+1}\to\mathbb{CP}^n$ is the Hopf map, and $h:\mathbb{CP}^n\to\mathbb{CP}^n$ is a holomorphic map.
Equivalently, factorization can be expressed by the following commutative diagram:
\eq{
\begin{CD}
\S^{2n+1} @<<\Xi< \S^{2n+1} \\
@VV \pi V @VV u V \\
\mathbb{CP}^n @> h >> \mathbb{CP}^n
\end{CD}
}
\end{theorem}

For the proof for the general $n\ge1$, we adapt Rivi\`{e}re's test vectors to show that any non-constant harmonic map $u:\S^{2n+1}\to\mathbb{CP}^n$ with the least index $2n+2$ satisfies \eqref{eq_qhwc}. However, unlike the case $n=1$, pseudo horizontal weak conformality does not necessarily imply horizontal weak conformality (see Appendix~\ref{appendix_B} for an example), hence the previous argument does not work. This is the typical difficulty for generalizing results of 2-dimensional targets to higher dimensional ones.
We therefore observe an additional key consequence of the index condition:
\eq{ \label{eq:key_observation}
    {\rm span}_\R\, \{\rd u(X_j) \,|\, 1 \leq j \leq 2n+2 \} = {\rm span}_\R\, \{ J \rd u(X_j)\,|\, 1 \leq j \leq 2n+2\}\text{\footnotemark},
}
\footnotetext{It was also observed by S. Egor in {\it Harmonic maps from spheres with lowest possible index and their properties},
{\bf arXiv:2609.35589v1}.}
where $\{X_j\}_{j=1}^{2n+2}$ denotes the conformal Killing vector fields on $\S^{2n+1}$. This identity, together with \eqref{eq_qhwc}, allows us to construct an isometry $\Xi$ and a standard Sasakian structure on $\S^{2n+1}$, hence the Hopf fibration $\pi$ arises (see Appendix~\ref{appendix_A} for the definition). The map $u$ then descends to $h:\mathbb{CP}^n\to\mathbb{CP}^n$, giving
\eq{
    u = h \circ \pi \circ \Xi.
}
Finally we show that $h$ is holomorphic. Since the isometric transformation $\Xi$ does not change any geometric quantities, in the sequel, we may assume that $\Xi$ is the identity map.

Then we prove 

\begin{theorem}\label{thmC.1}
Let $u=h\circ\pi:\S^{2n+1}\to\mathbb{CP}^n$, where $\pi:\S^{2n+1}\to\mathbb{CP}^n$ is the Hopf map and $h:\mathbb{CP}^n\to\mathbb{CP}^n$ is holomorphic of algebraic degree $d\ge 1 $. Then
\eq{
    {\rm ind}(u) \ge 2(n+1)\left\{ \binom{n+d-1}{n} + \binom{n+d-2}{n} + \cdots + \binom{n+d-(2n-1)}{n} \right\}
}
and
\eq{
    {\rm nul}(u) \ge 2(n+1)\left\{ \binom{n+d}{n} + \binom{n+d-2n}{n} \right\} - 2.
}
We use the convention $\binom{m}{n}=0$ for $m<n$.
\end{theorem}

Theorem \ref{thmC.1} implies that the holomorphic map $h$ obtained in Theorem \ref{Theorem_R} must have degree one. In fact, it follows that

\begin{corollary}
If $d\ge 2$, then ${\rm ind}(u)>2n+2$. Equivalently, if ${\rm ind}(u)\le 2n+2$, then $h$ has degree one.
\end{corollary}

Finally, we prove that the Hopf map $\pi:\S^{2n+1} \to \mathbb{CP}^n$ has index $2n+2$, which is a generalization of Theorem A.\text{\footnotemark},  \footnotetext{We just learned that this result was proved independently and posted early  by D. Gutwein and T. Langlais, Theorem 1.4 in {\it Deformations of harmonic maps with conical singularities}, {\bf arXiv:2609.20588}.   In the paper, the authors mentioned  that an upcoming work of  L. Lara and \'E. Loubeau, {\it Harmonic maps between 
$\S^{2n+1} $ and $\mathbb{CP}^n$
with low Morse index,} proves ``the analogous
result for all complex Hopf fibrations $\S^{2n+1} \to \mathbb{CP}^n$, and in particular recovering the value $2(n+ 1)$ for
the Morse index''.
}
In fact, we prove more, which is new even for $n=1$ as mentioned above, to the best of our knowledge.

\begin{theorem}\label{thm_nullity}
Let $u=h\circ\pi$ be the composition of the Hopf map $\pi$ with any degree-one holomorphic map $h:\CP^n\to\CP^n$. We have
\eq{
    {\rm ind}(u)=2n+2,\qquad {\rm nul}(u)=3n^2+5n.
}
\end{theorem}

\begin{remark}
Let us recover the isometry $\Xi$. We note that identity component of this family $u=h\circ\pi\circ\Xi$ is parametrized by the homogeneous space
\eq{\label{eq:sym_group}
    \frac{{\rm PGL}(n+1,\C)\times{\rm SO}(2n+2)}{{\rm U}(n+1)},
}
where ${\rm PGL}(n+1,\C)$ records the degree-one holomorphic map $h$, ${\rm SO}(2n+2)$ records the isometry $\Xi$, and ${\rm U}(n+1)$ cancels the repeated counting since the Hopf map $\pi$ is ${\rm U}(n+1)$-equivariant. The $\R$-dimension of \eqref{eq:sym_group} is precisely
\eq{
    2\{(n+1)^2-1\}+\frac{(2n+2)(2n+1)}{2}-(n+1)^2 = 3n^2+5n.
}
Hence Theorem~\ref{thm_nullity} implies that $h\circ\pi\circ \Xi$ is nondegenerate modulo its symmetry group. We record it as a corollary since it has its own interest.
\end{remark}

\begin{corollary}
    The map $h\circ \pi\circ \Xi$, as a harmonic map from $\S^{2n+1}$ to $\mathbb{CP}^n$, is nondegenerate modulo its symmetry group.
\end{corollary}

These three Theorems imply Theorem \ref{main_thm}.

We now emphasize the difference between the algebraic degree and the topological degree. Each holomorphic self-map of $\CP^n$ is formed by holomorphic homogeneous polynomials of algebraic degree $d$, see the discussion in Subsection~\ref{sec4.5}. We say that it has algebraic degree $d$, and topological degree $d^n$. When $n=1$, the two values coincide.

\ 

\noindent\textit{Organization of the paper.} Section~\ref{sec2} reviews the first and second variation formulas for harmonic maps, conformal Killing vector fields on spheres, the Weitzenb\"ock formula for differential forms, and basic facts about the Morse index, including El Soufi's lower bound for maps from spheres. Section~\ref{sec3} develops the test-vector argument (in the spirit of \cite{Riviere23}), proves \eqref{eq_qhwc} and \eqref{eq:key_observation}, and deduces the decomposition $u=h\circ\pi\circ\Xi$ in Theorem~\ref{Theorem_R}. Section~\ref{sec4} uses the formulation of complexification of the second variation, and shows the sources of Morse index and nullity, which leads to lower bounds of the both. Section~\ref{sec5} devotes itself to the value of ${\rm ind}(h\circ\pi)$. In Section~\ref{sec6}  we compute the value of ${\rm nul}(h\circ\pi)$.
In Appendix~\ref{appendix_A}, we recall the Sasakian structure and the Hopf fibration. Appendix~\ref{appendix_B} gives an example, showing that pseudo horizontal weak conformality does not necessarily imply horizontal weak conformality.

\section{Preliminaries}\label{sec2}

\subsection{The Dirichlet energy and its variation formulas}
Let $(M,g)$ and $(N,h)$ be two closed smooth manifolds and let $u:(M,g)\to (N,h)$ be a smooth map. Its differential $\rd u$ is a section of $T^*M\otimes u^*TN$, with $u^*TN$ the pull-back of $TN$ by $u$. Throughout, $\{e_j\}$ denotes a local orthonormal frame on $M$ and $\{e_\alpha\}$ a local frame on $N$. In local coordinates, one may write
\eq{
    \rd u = \rd u^\alpha \otimes (e_\alpha\circ u) \in \Gamma(T^*M\otimes u^*TN),
}
where $u^\alpha\in C^\infty(M)$ are the component functions.

The (Dirichlet) energy of $u$ is
\eq{
    E(u) = \frac{1}{2} \int_M \abs{\rd u}_{T^*M\otimes u^*TN}^2
    = \frac{1}{2} \int_M \rtr_g(u^*h).
}
Let $u_t$ be a smooth variation with $u_0=u$ and a variation vector field
$$
 w \coloneqq \frac{\rd u_t}{\rd t}\Big|_{t=0} \in \Gamma(u^*TN).
$$
The first variation formula (see \cite{Xin12}) is
\eq{
    \frac{\rd}{\rd t}\Big|_{t=0}E(u_t) = -\int_M \<\rtr_g(\nabla\rd u), w\>_h.
}
Accordingly, $u$ is harmonic if and only if its tension field vanishes:
\eq{\label{eq_harmonic_map}
    \tau_g(u) \coloneqq \rtr_g(\nabla\rd u) = (\nabla_{e_j}\rd u)(e_j)=0.
}
The second variation formula (see \cite{Xin12}) can be written in terms of the Jacobi operator; with the convention for the Riemann curvature $R$
$$
 R(X,Y)Z \coloneqq \nabla_X\nabla_Y Z-\nabla_Y\nabla_X Z-\nabla_{[X,Y]}Z,
$$
the Jacobi operator along a harmonic map is
\eq{\label{eq_Jacobi_operator}
    Lw \coloneqq \nabla^*\nabla w - R^N\bigl(w,\rd u(e_j)\bigr)\rd u(e_j).
}
Equivalently, for a variation with variational vector field $w$ one has
\eq{
    \frac{\rd^2}{\rd t^2}\Big|_{t=0}E(u_t)
    = \int_M \Big( \abs{\nabla w}_{T^*M\otimes u^*TN}^2 - \<R^N(w,\rd u(e_j))\rd u(e_j), w\>_h \Big).
}
We work primarily from an intrinsic viewpoint; for an extrinsic approach see \cite{Lin-Wang08}.

\subsection{The conformal Killing vector fields on \texorpdfstring{$\mathbb{S}^m$}{m-sphere}}

Let $(M,g)=(\S^m,g_{{\rm st}})$. The conformal transformations of $\S^m$ (modulo isometries) are given by
\eq{\label{eq_def_phi}
    \phi_a(x) \coloneqq (1-\abs{a}^2)\frac{x-a}{\abs{x-a}^2}-a, \quad a\in \mathbb{B}^{m+1}.
}
Their generating vector fields can be taken as
\eq{
    X_i \coloneqq \nabla x_i = \epsilon_i - \<\epsilon_i, x\>x = \epsilon_i - x_ix,
}
where $\epsilon_i$ are the standard basis vectors in $\R^{m+1}$ and $x_i$ are the coordinate functions restricted to $\S^m$. $X_i$'s span the space of (non-trivial) conformal Killing vector fields, which plays a very important role in this paper. 

The functions $\{x_i\}_{i=1}^{m+1}$ span the first eigenspace of the rough Laplacian $\nabla^*\nabla$ with eigenvalue $m$; more precisely, for $f=x_i$ one has
\eq{
    \nabla^*\nabla f = mf, \quad {\rm Hess}f = -fg.
}
In particular,
\eq{\label{eq_conformal_Killing}
    \nabla_X\nabla f = {\rm Hess}f(X) = -fX, \quad \forall X\in\Gamma(T\S^m),
}
so $\nabla f$ is a conformal Killing vector field on $\S^m$.

\subsection{The Weitzenb\"ock formula}

The Hodge Laplacian on differential forms is
\eq{
    \Delta = (\rd + \rd^*)^2 = \rd\rd^* + \rd^*\rd,
}
where $\rd$ is the exterior derivative and $\rd^*$ the codifferential. It is related to the rough Laplacian $\nabla^*\nabla$ by the Weitzenb\"ock formula
\eq{
    \Delta = \nabla^*\nabla + \sum_{j,k} e_k^\flat \w e_j \ip R(e_j, e_k),
}
where $\nabla^*\nabla = -\nabla_{e_i}\nabla_{e_i}$ in a geodesic coordinate. For a map $u:(M,g)\to (N,h)$ we regard $\rd u$ as a $u^*TN$-valued $1$-form, i.e. $\rd u\in \Omega^1(M,u^*TN)$. In this setting, the Weitzenb\"ock formula reads
\eq{\label{eq_Weitzenbock}
    \Delta\rd u = \nabla^*\nabla\rd u + {\rm Ric}^M(\rd u) - R^N\bigl(\rd u, \rd u(e_j)\bigr)\rd u(e_j).
}

\subsection{Morse index of a harmonic map}

For a harmonic map $u:(M,g)\to (N,h)$, the Morse index ${\rm ind}(u)$ (with respect to the Dirichlet energy) is the maximal dimension of a subspace of $ \Gamma(u^*TN)$ on which the index form is negative definite. The index form is
\eq{
    I(w,v) \coloneqq \int_M \<Lw,v\>.
}

As recalled in the Introduction, El Soufi \cite{ElSoufi95} proved that any non-constant harmonic map from $\S^m$ into an arbitrary target has index at least $m+1$; see also Ni \cite{Ni23}. For completeness and also for fixing notations, we include a proof in intrinsic notation.

\begin{theorem}[\cite{ElSoufi95}]\label{thm_2.1}
Let $u$ be a smooth non-constant harmonic map from $\S^m$ ($m\geq3$) to any target $N$. Then ${\rm ind}(u)\ge m+1$.
\end{theorem}

\begin{proof}
Fix $f=x_i$ and consider the conformal Killing field $\nabla f$ on $\S^m$. Set
$$
    w \coloneqq \rd u(\nabla f) \in \Gamma(u^*TN).
$$
At a point $p\in\S^m$, choose a local orthonormal frame $\{e_j\}$ with $\nabla e_j|_p=0$. Using \eqref{eq_conformal_Killing}, we obtain
\eq{\label{eq0_thm2.1}
    \nabla_{e_j}w
    = (\nabla_{e_j}\rd u)(\nabla f) + \rd u(\nabla_{e_j}\nabla f)
    = (\nabla_{e_j}\rd u)(\nabla f) - f\rd u(e_j)
}
and
\eq{\label{eq1_thm2.1}
    \nabla_{e_j}\nabla_{e_j}w
    &= (\nabla_{e_j}\nabla_{e_j}\rd u)(\nabla f) + (\nabla_{e_j}\rd u)(\nabla_{e_j}\nabla f) - (\nabla_{e_j}f)\rd u (e_j) - f(\nabla_{e_j}\rd u)(e_j)\\
    &= (-\nabla^*\nabla\rd u)(\nabla f) - f(\nabla_{e_j}\rd u)(e_j) - \rd u(\nabla f) - f(\nabla_{e_j}\rd u)(e_j)\\
    &= (-\nabla^*\nabla\rd u)(\nabla f) - w - 2f(\nabla_{e_j}\rd u)(e_j).
}
Since ${\rm Ric}^{\S^m}=(m-1)g$, \eqref{eq_Weitzenbock} becomes
\eq{\label{eq:eq2a_Thm2.1}
    -\nabla^*\nabla\rd u = -\Delta\rd u + (m-1)\rd u - R^N(\rd u,\rd u(e_j))\rd u(e_j).
}
Moreover, the harmonic map equation \eqref{eq_harmonic_map} gives
\eq{
    0 = (\nabla_{e_j}\rd u)(e_j) = -\rd^*\rd u,
}
and hence
\eq{
    \Delta\rd u = \rd\rd^*\rd u = 0.
}
Substituting the above identities into \eqref{eq1_thm2.1} and \eqref{eq:eq2a_Thm2.1} yields
\eq{\label{eq2_thm2.1}
    -\nabla^*\nabla w = \nabla_{e_j}\nabla_{e_j}w = (m-2)w - R^N(w,\rd u(e_j))\rd u(e_j).
}
Therefore, by \eqref{eq_Jacobi_operator},
\eq{\label{eq3_thm2.1}
    Lw = -(m-2)w.
}

The index estimate now follows from the fact (proved in \cite{ElSoufi95}) that the sections $\{\rd u(X_i)\,|\,1\le i\le m+1\}$ are linearly independent; this is a consequence of Lemma~\ref{lem2.1} below.
\end{proof}

\begin{lemma}[\cite{ElSoufi95}]\label{lem2.1}
For any non-constant map $u:\S^m\to N$ ($m\geq3$) and any non-trivial conformal Killing vector field $X$ on $\S^m$, one has $\rd u(X)\neq 0$.
\end{lemma}

\begin{proof}
Suppose that $X$ is a conformal Killing field with $\rd u(X)=0$ on $\S^m$. Without loss of generality, assume $X=X_1=\nabla x_1$. Then $u\circ\phi_{t\epsilon_1}$ is independent of   $t\in(0,1)$, where $\phi$ is defined in \eqref{eq_def_phi}. Letting $t\to 1$ forces $u$ to be constant, a contradiction.
\end{proof}

\section{Proof of Theorem \ref{Theorem_R}}\label{sec3}

We start with an algebraic consequence of the least-index assumption. For $n=1$ this observation is due to Rivi\`{e}re~\cite{Riviere23}.

\begin{proposition}\label{prop3.1}
Let $u:(\S^{2n+1},g_{\rm st})\to(\mathbb{CP}^n,h)$ be a smooth harmonic map, where $h$ is the Fubini--Study metric. If ${\rm ind}(u)\le 2n+2$, then
\eq{\label{commutative_dudu_J}
[(\rd u)(\rd u)^*,J]=0,
}
where $J$ is the complex structure on $\mathbb{CP}^n$.
\end{proposition}

\begin{proof}
We follow \cite{Riviere23}.
If $u$ is constant there is nothing to prove. Assume $u$ is non-constant. By Theorem~\ref{thm_2.1} we have ${\rm ind}(u)\ge 2n+2$, so necessarily ${\rm ind}(u)=2n+2$.

From the proof of Theorem~\ref{thm_2.1}, the sections $\{\rd u(X_i)\}_{i=1}^{2n+2}$ are linearly independent and satisfy
\eq{
    L\,\rd u(X_i) = -(2n-1)\rd u(X_i), \quad 1\le i\le 2n+2.
}
Thus $-(2n-1)$ is the unique negative eigenvalue of $L$, with eigenspace spanned by $\{\rd u(X_i)\}_{i=1}^{2n+2}$. Equivalently,
\eq{\label{eq0_thm1.1}
    L + (2n-1) \text{ is positive semi-definite}.
}
We now test \eqref{eq0_thm1.1} against $Jw$ with $w\coloneqq\rd u(X_i)$. Since $\nabla J=0$ and using \eqref{eq2_thm2.1}, we obtain
\eq{
    -\nabla ^* \nabla (Jw) = (2n-1) Jw - J (R^N(w, \rd u(e_j))\rd u(e_j)),
}
where $N=\mathbb{CP}^n$. Using \eqref{eq_Jacobi_operator} we have
\eq{\label{eq_eigen_J}
    (L+2n-1)(Jw) = J (R^N(w, \rd u(e_j))\rd u(e_j)) - R^N(Jw,\rd u(e_j))\rd u(e_j).
}
We compute the right-hand side by recalling that the Riemann curvature of $\mathbb{CP}^n$ satisfies (see for instance \cite{Kobayashi-Nomizu96}, and note that we adopt the convention that the holomorphic sectional curvature is $4$ in order to make the Hopf map a Riemannian submersion)
\eq{\label{holo_sec_curvature}
    R(X,Y)Y = \abs{Y}^2X - \<X,Y\>Y + 3\<X,JY\>JY.
}
It follows
\eq{
    J(R^N(w, \rd u(e_j))\rd u(e_j)) = J(\abs{\rd u}^2w - \<w,\rd u(e_j)\>\rd u(e_j) + 3\<w, J\rd u(e_j)\>J\rd u(e_j) )
}
and
\eq{
    R^N(Jw, \rd u(e_j))\rd u(e_j) = \abs{\rd u}^2 Jw - \<Jw, \rd u(e_j)\>\rd u(e_j) + 3\<w, \rd u(e_j)\>J\rd u(e_j).
}
Putting them into the right-hand side of \eqref{eq_eigen_J} we have
\eq{
    \frac{1}{4}(L+2n-1)(Jw)=-\< w, \rd u(e_j)\> J\rd u(e_j) + \< Jw, \rd u(e_j)\> \rd u(e_j),
}
and hence
\eq{\label{eq1_thm1.1}
    \frac{1}{4}\int\<(L+2n-1)(Jw),Jw\> = \int\sum_j \<\rd u(e_j),Jw\>^2 - \int\sum_j \<\rd u(e_j),w\>^2 = \int\<\rd u,Jw\>^2 - \int\<\rd u,w\>^2,
}
which is non-negative, by \eqref{eq0_thm1.1}.
Choosing $w=\rd u(X_i)$ in \eqref{eq1_thm1.1} and summing over $i$ yields
\eq{\label{eq_sum}
\frac 14 \int\sum_{i=1}^{2n+2}\<(L+2n-1) (J\rd u(X_i)), J\rd u (X_i)\>=\int\sum_{i=1}^{2n+2}\{\< \rd u, J \rd u(X_i)\>^2 -\<\rd u, \rd u(X_i)\>^2\}=:\int I.
}
We claim that
\eq{\label{claim}
    I=-\frac{1}{2}\|[(\rd u) (\rd u)^*, J]\|_{\rm HS}^2,
}
where $[(\rd u) (\rd u)^*, J]=(\rd u) (\rd u)^*J-J(\rd u)(\rd u)^*:u^*T\mathbb{CP}^n \to u^*T\mathbb{CP}^n$.\, and $\|\cdot\|_{\mathrm{HS}}$ denotes the Hilbert--Schmidt norm on endomorphisms of $u^*T\mathbb{CP}^n$.
Since the integration of the left hand side
of \eqref{eq_sum} is non-negative, the Proposition follows clearly from the claim.

We now prove the claim. In the following computation, it is convenient to use $\{X_i\}_{i=1}^{2n+2}$ instead of $\{e_j\}_{i=1}^{2n+1}$, though  $\{X_i\}_{i=1}^{2n+2}$ is not a basis.
Note that for the standard metric on $\S^{2n+1}$ we have
\eq{
    g_{{\rm st}} = \sum_{i=1}^{2n+2} \rd x_i\otimes \rd x_i = \sum_{i=1}^{2n+2} X_i^\flat\otimes X_i^\flat.
}
Therefore for any 1-forms $\alpha,\beta$ on $\S^{2n+1}$ we have
\eq{\label{X_i_is_basis}
    \<\alpha,\beta\> = \sum_{j=1}^{2n+1} \alpha(e_j)\beta(e_j)=\sum_{i=1}^{2n+2} \alpha(X_i)\beta(X_i).
}
It follows
\eq{
    \< \rd u , Jw\>^2=\sum _{j} \< \rd u(e_j), Jw\>^2=\sum_{i=1}^{2n+2}\< \rd u (X_i), J w\>^2. 
} 
Similarly, we can compute $\<\rd u,w\>^2$.
Therefore, 
\eq{\label{eq2_thm1.1}
  I=\sum_{i,j=1}^{2n+2} \Big( \<\rd u(X_i), J\rd u(X_j)\>^2 - \<\rd u(X_i),\rd u(X_j)\>^2 \Big).
}
To continue, we choose a local orthonormal basis on $\mathbb{CP}^n$ by $\{e_\alpha,e_{\bar{\alpha}}\coloneqq Je_\alpha\}_{\alpha=1}^n$, and denote
\eq{
    u_i^\alpha \coloneqq \rd u^\alpha(X_i), \quad u_i^{\bar{\alpha}} \coloneqq \rd u^{\bar{\alpha}}(X_i).}
It follows
\begin{align}
    I&=\sum_{i,j} \Big( \<\rd u(X_i),J\rd u(X_j)\>^2 - \<\rd u(X_i),\rd u(X_j)\>^2 \Big)\\
    &= \sum_{i,j} \Big\{ \Big( \sum_\alpha u_i^{\bar{\alpha}}u_j^{\alpha} - \sum_\alpha u_i^{\alpha}u_j^{\bar{\alpha}} \Big)^2 - \Big( \sum_\alpha u_i^{\alpha}u_j^{\alpha} + \sum_\alpha u_i^{\bar{\alpha}}u_j^{\bar{\alpha}} \Big)^2 \Big\}\\
    &= \sum_{i,j} \Big\{ \Big( \sum_\alpha u_i^{\bar{\alpha}}u_j^{\alpha} \Big)^2 + \Big( \sum_\alpha u_i^{\alpha}u_j^{\bar{\alpha}} \Big)^2 - 2\Big( \sum_\alpha u_i^{\bar{\alpha}}u_j^{\alpha} \Big)\Big( \sum_\alpha u_i^{\alpha}u_j^{\bar{\alpha}} \Big)\\
    &\quad\qquad -\Big( \sum_\alpha u_i^{\alpha}u_j^{\alpha} \Big)^2 - \Big( \sum_\alpha u_i^{\bar{\alpha}}u_j^{\bar{\alpha}} \Big)^2 - 2\Big( \sum_\alpha u_i^{\alpha}u_j^{\alpha} \Big)\Big( \sum_\alpha u_i^{\bar{\alpha}}u_j^{\bar{\alpha}} \Big) \Big\}\\
    &= \sum_{i,j} \Big\{ \sum_{\alpha,\beta} u_i^{\bar{\alpha}}u_j^\alpha u_i^{\bar{\beta}}u_j^\beta +\sum_{\alpha,\beta} u_i^{\alpha}u_j^{\bar{\alpha}} u_i^{\beta}u_j^{\bar{\beta}} - 2\sum_{\alpha,\beta} u_i^{\bar{\alpha}}u_j^{\alpha} u_i^{\beta}u_j^{\bar{\beta}}\\
    &\quad\qquad -\sum_{\alpha,\beta} u_i^{\alpha}u_j^{\alpha} u_i^{\beta}u_j^{\beta} - \sum_{\alpha,\beta} u_i^{\bar{\alpha}}u_j^{\bar{\alpha}} u_i^{\bar{\beta}}u_j^{\bar{\beta}} - 2\sum_{\alpha,\beta} u_i^{\alpha}u_j^{\alpha} u_i^{\bar{\beta}}u_j^{\bar{\beta}} \Big\}\\
    &= \sum_{i,j} \Big\{ -\sum_{\alpha,\beta} \big(u_i^\alpha u_i^\beta - u_i^{\bar{\alpha}}u_i^{\bar{\beta}} \big) \big(u_j^\alpha u_j^\beta - u_j^{\bar{\alpha}}u_j^{\bar{\beta}} \big) -\sum_{\alpha,\beta} \big(u_i^\alpha u_i^{\bar{\beta}} + u_i^{\bar{\alpha}}u_i^{\beta} \big) \big(u_j^\alpha u_j^{\bar{\beta}} + u_j^{\bar{\alpha}}u_j^{\beta} \big) \Big\}\\
    &= -\sum_{\alpha,\beta} \Big\{ \sum_i \big( u_i^\alpha u_i^\beta - u_i^{\bar{\alpha}}u_i^{\bar{\beta}} \big) \Big\}^2 - \sum_{\alpha,\beta} \Big\{ \sum_i \big(u_i^\alpha u_i^{\bar{\beta}} + u_i^{\bar{\alpha}}u_i^{\beta} \big) \Big\}^2
    = -\frac12\,\big\|[(\rd u)(\rd u)^*,J]\big\|_{\mathrm{HS}}^2.
\end{align}
To see the last equality, set $A\coloneqq (\rd u)(\rd u)^*$ and use the orthonormal frame $\{e_\alpha,e_{\bar\alpha}=Je_\alpha\}$. Then
\[
\<Ae_\alpha,e_\beta\>=\sum_i u_i^\alpha u_i^\beta,\quad \<Ae_{\bar\alpha},e_{\bar\beta}\>=\sum_i u_i^{\bar\alpha}u_i^{\bar\beta},\quad \<Ae_\alpha,e_{\bar\beta}\>=\sum_i u_i^\alpha u_i^{\bar\beta}.
\]
Writing matrices in the ordered basis $(e_1,\dots,e_n,e_{\bar1},\dots,e_{\bar n})$, we have
$J=\begin{psmallmatrix}0&-I\\ I&0\end{psmallmatrix}$ and $A=\begin{psmallmatrix}P&Q\\ Q^\top&R\end{psmallmatrix}$ with
$P_{\alpha\beta}=\sum_i u_i^\alpha u_i^\beta$, $R_{\alpha\beta}=\sum_i u_i^{\bar\alpha}u_i^{\bar\beta}$, $Q_{\alpha\beta}=\sum_i u_i^\alpha u_i^{\bar\beta}$.
A direct block computation gives
$[A,J]=AJ-JA=\begin{psmallmatrix}Q+Q^\top&R-P\\ R-P&-(Q+Q^\top)\end{psmallmatrix}$,
so
\[
\|[A,J]\|_{\mathrm{HS}}^2
=2\sum_{\alpha,\beta}(P_{\alpha\beta}-R_{\alpha\beta})^2+2\sum_{\alpha,\beta}(Q_{\alpha\beta}+Q_{\beta\alpha})^2,
\]
which implies the last equality.
\end{proof}

\begin{remark}
(1) If $v\in\Gamma(u^*T\mathbb{CP}^n)$ is a $\lambda$-eigenvector of $(\rd u)(\rd u)^*$, then \eqref{commutative_dudu_J} implies that $Jv$ is also a $\lambda$-eigenvector. Hence each eigenspace of $(\rd u)(\rd u)^*$ is $J$-invariant and, in a suitable orthonormal basis,
\eq{\label{eq:eigenvalue}
    (\rd u)(\rd u)^* = \mathrm{diag}\{\lambda_1,\lambda_1,\lambda_2,\lambda_2,\dots,\lambda_n,\lambda_n\},
}
with $\lambda_i\ge 0$. In particular, $\mathrm{rank}(\rd u)$ is even at every point.

(2) The commutation relation \eqref{commutative_dudu_J} goes back to Burns--de~Bartolomeis~\cite{Burns-deB88} and Burns--Burstall--de~Bartolomeis--Rawnsley~\cite{BBdBR89} (in the stable case for irreducible compact Hermitian symmetric spaces); see also Chen~\cite{Chen96}.

(3) A map satisfying \eqref{commutative_dudu_J} (equivalently \eqref{eq:eigenvalue}) is called \emph{pseudo horizontally weakly conformal} in \cite{Loubeau97}. If, moreover, all $2n$ eigenvalues coincide at each point, then $u$ is \emph{horizontally weakly conformal} (a strictly stronger condition for $n>1$).
\end{remark}

For $n=1$, \eqref{eq:eigenvalue} reads $(\rd u)(\rd u)^*=\mathrm{diag}\{\lambda_1,\lambda_1\}$, i.e. $u$ is horizontally weakly conformal. Combined with harmonicity, this implies that $u$ is a harmonic morphism; the classification of harmonic morphisms $\S^3\to\S^2$ \cite{Baird-Wood88} then yields Theorem~B (as in \cite{Riviere23}).

For $n>1$, pseudo horizontal weak conformality does \emph{not} imply horizontal weak conformality in general, so $u$ need not be a harmonic morphism (see Appendix). Instead, we will combine \eqref{commutative_dudu_J} with an additional rigidity property to obtain the factorization \eqref{eq0}.

\begin{proposition} \label{prop3.3} 
If ${\rm ind}(u)=2n+2$, then 
\eq{
    (L+2n-1)(J\rd u(X_i)) = 0, \quad\forall\, 1\leq i\leq 2n+2.
}
Moreover, we have
\eq{
    {\rm span}_\R \{\rd u(X_j) \,|\, 1 \leq j \leq 2n+2 \} = {\rm span}_\R \{ J \rd u(X_j)\,|\, 1 \leq j \leq 2n+2\}.
}

\end{proposition}
\begin{proof}
From the proof of Proposition~\ref{prop3.1} we see that $-(2n-1)$ is the only negative eigenvalue of $L$, and the eigenspace is spanned by $\{\rd u(X_i) \,|\, 1\leq i\leq 2n+2\}$. Moreover, the proof of Proposition~\ref{prop3.1}
implies that
\eq{
    \int_{\S^{2n+1}} \<(L+2n-1)(J\rd u(X_i)),J\rd u(X_i)\> = 0, \quad\forall\, 1\leq i\leq 2n+2.
}
Since $L+2n-1$ is symmetric and semi-positive definite (see \eqref{eq0_thm1.1}), we have
\eq{
    (L+2n-1)(J\rd u(X_i)) = 0, \quad\forall\, 1\leq i\leq 2n+2.
}
Therefore all $J(\rd u(X_i))$'s are also $-(2n-1)$-eigenvectors of $L$. The  linearly independence of $J(\rd u(X_i))$'s follows from that of $\rd u(X_i)$'s.
Hence the conclusion follows.
\end{proof}

We now prove Theorem~\ref{Theorem_R}.

\begin{proof}[Proof of Theorem~\ref{Theorem_R}]
As a consequence of Proposition~\ref{prop3.3}, there exists a constant $(2n+2)\times(2n+2)$ matrix $A$ such that
\eq{\label{eq:jj}
    J\rd u(X_i) = A_{ij}\rd u(X_j) = \rd u(A_{ij}X_j).
}

\noindent{\bf Step 1.} We show $A^2=-I$, where $I$ is the identity matrix. 
From \eqref{eq:jj} we have 
\eq{
    -\rd u(X_i) = J^2\rd u(X_i) = J\rd u(A_{ij}X_j) = \rd u(A_{ij}A_{jk}X_k),
}
therefore
\eq{
    \rd u(X_i + A_{ij}A_{jk}X_k) = 0.
}
Lemma~\ref{lem2.1} implies
\eq{
    X_i + A_{ij}A_{jk}X_k = 0.
}
Since $A$ is constant, this implies
\eq{\label{eq:0}
    A^2 = -I.
}

\noindent{\bf Step 2.} We prove that $A^*+A=0$.

Fix $x\in\S^{2n+1}$. Extend $\rd u|_x:T_x\S^{2n+1}\to T_{u(x)}\mathbb{CP}^n$ to a linear operator
\eq{\label{def_B_x}
    B_x : \R^{2n+2} \to T_{u(x)}\mathbb{CP}^n
}
by
\eq{
    B_x(y) \coloneqq y_i\rd u|_x(X_i) = \rd u|_x(y_iX_i) = \rd u|_x(y^\top),
}
for $y=y_i\epsilon_i$. In other words,
\eq{\label{eq:0.8}
    B_x = \rd u|_x \circ T,
}
where $T$ is the projection onto the tangent space of $\S^{2n+1}$.

It is clear that $x \in {\rm ker}(B_x)$.
Indeed, we have
\eq{\label{eq:0.9}
    B_x(x) = B_x(x_i\epsilon_i) = x_i\rd u(X_i) = \rd u(x_iX_i) = 0.
}
By definition of $A$, we have for any $y\in\R^{2n+2}$
\eq{
    B_x(A^*y) = (A^*y)_j\rd u(X_j) = A_{ij}y_i\rd u(X_j) = y_iJ\rd u(X_i) = JB_x(y).
}
Hence
\eq{\label{eq:1}
    B_x A^* = J B_x \qquad\hbox{on}\quad \R^{2n+2}.
}
Since $J^*=-J$, taking adjoint gives
\eq{\label{eq:2}
    A B_x^* = -B_x^* J.
}
Note that
\eq{
    B_xB_x^* = (\rd u)_x(\rd u)_x^*.
}
Using \eqref{eq:2} and \eqref{commutative_dudu_J}, we obtain
\eq{
    B_x(A+A^*)B_x^*
    = B_xAB_x^* + B_xA^*B_x^*
    = -B_xB_x^*J + JB_xB_x^*
    = -\big[(\rd u)_x(\rd u)_x^*,J\big]
    = 0.
}
Since ${\rm im}(B_x^*)={\rm ker}(B_x)^\perp$, it follows that
\eq{\label{eq:2.1}
    \<(A+A^*)y,z\>=0, \quad\forall\,y,z\in {\rm ker}(B_x)^\perp.
}

\eqref{eq:0} implies that $(A^*)^2=-I$, hence we may identify $\R^{2n+2}$ with $\C^{n+1}$ by viewing $A^*$ as a complex structure $J_A$.
Denote by $\CP^n_A$ the associated complex projective space, namely
\eq{
    \CP^n_A \coloneqq \mathbb{P}(\R^{2n+2},A^*),
}
Denote the associated quotient map by
\eq{
    \pi_A:\S^{2n+1}\to\CP^n_A.
}
Define
\eq{
    \xi_x \coloneqq (A^*x)^\top = (A^*x)_jX_j.
}
It follows from \eqref{eq:0.9} that
\eq{\label{eq:3}
    \rd u|_x(\xi_x) = (A^*x)_j\rd u|_x(X_j) = B_x(A^*x) = JB_x(x) = 0.
}
Note that the $\pi_A$-fibre at each point $[x]\in\CP^n_A$ is $\S^{2n+1}\cap{\rm span}_\R\{x,A^*x\}$, which is a connected space. Hence \eqref{eq:3} implies that $u$ is constant on each $\pi_A$-fibre. Therefore, $u$ descends to a map $h:\CP^n_A\to\CP^n$, namely $u=h\circ\pi_A$.

We claim that $h$ is holomorphic. In fact, for any $v\in T_x\S^{2n+1}$ we have
\eq{
    J\rd u(v) = v_iJ\rd u(X_i) = v_iA_{ij}\rd u(X_j) = (A^*v)_j\rd u(X_j) = \rd u((A^*v)^\top),
}
and
\eq{
    \rd\pi_A((A^*v)^\top) = J_A\rd\pi_A(v).
}
It follows that
\eq{
    J\circ\rd h\circ\rd\pi_A(v) = J\circ\rd u(v) = \rd u((A^*v)^\top) = \rd h \circ\rd\pi_A((A^*v)^\top) = \rd h \circ J_A\circ\rd\pi_A(v).
}
Since $\rd\pi_A$ is surjective, we see that
\eq{
    J\circ\rd h =\rd h \circ J_A,
}
the claim.

Since $h$ is holomorphic, and $\deg(h)>0$, its complex Jacobian is not identically zero. Hence $h$ must have a full complex rank $n$ on a dense open set. Namely,
\eq{
    \mathcal{R}\coloneqq \{x\in\S^{2n+1} \mid {\rm rank}_\R(\rd u|_x)=2n \}
}
is a dense open set. In view of \eqref{eq:0.8}, we have
\eq{
    {\rm rank}_\R(B_x) = {\rm rank}_\R(\rd u|_x) = 2n, \quad\forall\,x\in\mathcal{R}.
}
Hence
\eq{
    {\rm dim}_\R\ker(B_x)=2, \quad\forall\,x\in\mathcal{R}.
}
By \eqref{eq:0.9} and \eqref{eq:3}, we see that $x,A^*x\in\ker(B_x)$, and hence
\eq{
    \ker(B_x)={\rm span}_\R\{x,A^*x\}, \quad\forall\,x\in\mathcal{R}.
}
Then \eqref{eq:2.1} implies that
\eq{\label{eq:4}
    \<(A+A^*)y,z\>=0, \quad\forall\,y,z\in {\rm span}_\R\{x,A^*x\}^\perp, \quad\forall\,x\in\mathcal{R}.
}
If we denote by $P_x$ the projection onto ${\rm span}_\R\{x,A^*x\}^\perp$, then \eqref{eq:4} becomes
\eq{
    P_x\circ (A+A^*)\circ P_x = 0,\quad\forall\,x\in\mathcal{R}.
}
By continuity, we have
\eq{\label{eq:5}
    P_x\circ (A+A^*)\circ P_x = 0,\quad\forall\,x\in\S^{2n+1}.
}
Finally, for any $y\in\R^{2n+2}$, there exists $0\neq x\in{\rm span}_\R\{y,Ay\}^\perp$. Hence
\eq{
    \<y,x\>=0, \qquad \<y,A^*x\>=\<Ay,x\>=0,
}
which means $y\in{\rm span}_\R\{x,A^*x\}^\perp$. Therefore, $P_x\,y=y$, and \eqref{eq:5} then implies that
\eq{
    \<(A+A^*)y,y\> = \<P_x(A+A^*)P_x\,y,y\> = 0,\quad\forall\,y\in\R^{2n+2}.
}
Hence $A+A^*=0$.

\noindent{\bf Step 3.}
By Steps 1--2, $A^*$ is an orthogonal complex structure on $\R^{2n+2}$. Since ${\rm O}(2n+2)$ acts transitively on the space of orthogonal complex structures, after composing with an isometry of $\S^{2n+1}$ we may assume that $A^*$ is the standard complex structure. Then $\CP^n_A=\CP^n$, and
\eq{
    \pi_A = \pi\circ\Xi, \qquad \Xi\in{\rm Iso}(\S^{2n+1}).
}
Therefore
\eq{
    u= h \circ\pi\circ\Xi.
}
The claim follows.
\end{proof}

Since the isometric transformation $\Xi$ does not change any geometric quantities, in the sequel, we may assume $\Xi$ is the identity map.

\section{Proof of Theorem \ref{thmC.1}}\label{sec4}

We now show that the holomorphic factor $h$ must have degree $1$. Indeed, we prove a much stronger Theorem~\ref{thmC.1}, which gives quantitative lower bounds for the index and nullity of $h\circ\pi$ expressed in terms of the algebraic degree of the holomorphic map $h$.

\subsection{Preliminaries on Sasakian manifolds and the Hopf fibration}
\label{subsec:sasakian_prelim}

We briefly recall the Sasakian geometry used below; see Appendix~\ref{appendix_A} for further background.

Let $(M^{2n+1},g)$ be a Riemannian manifold. A \emph{Sasakian structure} is a pair $(\xi,\Phi)$, where $\xi$ is a unit vector field (the \emph{Reeb field}) and $\Phi$ is a $(1,1)$-tensor, such that
\eq{
    \Phi^2(X)=-X+g(X,\xi)\xi,\qquad g(\Phi X,\Phi Y)=g(X,Y)-g(X,\xi)g(Y,\xi),
}
for all $X,Y\in\Gamma(TM)$ and
\eq{\label{eq:sasakian_def}
    (\nabla_X\Phi)(Y)=-g(X,Y)\xi+g(Y,\xi)X.
}
Set the contact $1$-form $\eta\coloneqq g(\xi,\cdot)$ and the horizontal distribution $\mathcal{H}\coloneqq\ker\eta$. Then $\Phi$ preserves $\mathcal{H}$ and induces an almost complex structure on $\mathcal{H}$. In particular, $(\mathcal{H},\Phi|_{\mathcal{H}},g|_{\mathcal{H}})$ is transversely K\"ahler, with transverse K\"ahler form
\eq{
    \omega^\top\coloneqq \tfrac12\,\rd\eta.
}

In our application, $\S^{2n+1}\subset\C^{n+1}$ carries the standard Sasakian structure induced by the standard complex structure $\mathbf{j}$ on $\C^{n+1}$:
\eq{
    \xi\coloneqq \mathbf{j}x,\qquad \Phi(X)\coloneqq \mathbf{j}X+\<X,\xi\>x,\qquad X\in\Gamma(T\S^{2n+1}).
}
The associated Hopf fibration
\eq{
    \S^1\to\S^{2n+1}\xrightarrow{\pi}\CP^n
}
is a Riemannian submersion with vertical space $\mathrm{span}\{\xi\}$ and horizontal space $\mathcal{H}$. In particular, for each $x\in\S^{2n+1}$ the differential restricts to an isometry
\eq{\label{eq:dp_isometry}
    \rd\pi_x\colon (\mathcal{H}_x,g|_{\mathcal{H}_x}) \longrightarrow (T_{\pi(x)}\CP^n,g_{FS}),
}
so $g|_{\mathcal{H}}=\pi^*g_{FS}$. Moreover, if $J$ denotes the complex structure on $\CP^n$, then
\eq{\label{eq:pi_intertwines}
    \rd\pi\circ\Phi = J\circ\rd\pi \qquad\text{on }\mathcal{H}.
}
Equivalently, $\omega^\top=\pi^*\omega_{FS}$ (with our normalizations). Thus the transverse K\"ahler geometry of $(\S^{2n+1},\xi,\Phi)$ is identified with the K\"ahler geometry of $(\CP^n,g_{FS})$ via $\pi$. For more information of the Sasakian structure and the Hopf map, see Appendix~\ref{appendix_A} below.

\subsection{The complexification}
Let $(M,J)$ be a K\"ahler manifold. Since $J^2=-\rid$, the complexified tangent bundle splits as
\eq{
    TM\otimes\C = T^{1,0}M\oplus T^{0,1}M.
}
Given a real vector field $w$, set
\eq{
    V\coloneqq \tfrac12(w-iJw),\qquad \bar V\coloneqq \tfrac12(w+iJw),
}
so that $w=V+\bar V$ and
\eq{\label{eq:JV}
    JV=iV,\qquad J\bar V=-i\bar V.
}
Throughout, $\<\cdot,\cdot\>$ denotes the complex-bilinear extension of the Riemannian metric (not the Hermitian pairing). In particular,
\eq{
    \<V,V\>=\<\bar V,\bar V\>=0,\qquad \<V,\bar V\>=\abs{V}^2,\qquad V\in\Gamma(T^{1,0}M).
}

The curvature tensor extends complex-multilinearly and satisfies
\eq{\label{eq_curvature_tensor_additional}
    R(X,Y,JZ,JW)=R(X,Y,Z,W),\qquad R(JX,Y,Z,W)=-R(X,JY,Z,W).
}
Consequently, the only potentially nonzero components are of type
\eq{\label{eq_only_non-zero}
    R(A,\bar B,C,\bar D),\qquad A,B,C,D\in\Gamma(T^{1,0}M),
}
up to the curvature symmetries and complex conjugation. For $(\CP^n,g_{FS})$ we use the standard curvature identity
\eq{
    R(X,Y)Z = \<Y,Z\>X - \<X,Z\>Y + \<JY,Z\>JX - \<JX,Z\>JY + 2\<X,JY\>JZ,
}
so, in particular,
\eq{
    R(A,\bar A,B,\bar B)=-2\big(\abs{A}^2\abs{B}^2+\abs{\<A,\bar B\>}^2\big),
    \qquad A,B\in\Gamma(T^{1,0}\CP^n).
}

Fix a local unitary frame $\{W_\alpha\}_{\alpha=1}^n$ of $T^{1,0}\CP^n$ on an open set $U\subset\CP^n$. Over $\pi^{-1}(U)$, let $Z_\alpha\in\Gamma(\mathcal{H}\otimes\C)$ be the \emph{horizontal lift} determined by
\eq{\label{eq:horizontal_lift}
    \rd\pi(Z_\alpha)=W_\alpha,\qquad Z_\alpha\in\mathcal{H}\otimes\C.
}
Then $\{Z_\alpha,\bar Z_\alpha\}_{\alpha=1}^n$ is unitary for $g|_{\mathcal{H}}$ (equivalently for $\pi^*g_{FS}$). Define the associated real orthonormal frame of $\mathcal{H}$ by
\eq{\label{eq:real_frame_from_complex}
    e_\alpha\coloneqq \tfrac{1}{\sqrt2}(Z_\alpha+\bar Z_\alpha),
    \qquad
    \Phi e_\alpha\coloneqq \tfrac{i}{\sqrt2}(Z_\alpha-\bar Z_\alpha).
}
Then $\{e_\alpha,\Phi e_\alpha,\xi\}$ is an adapted orthonormal frame on $\S^{2n+1}$ and
\eq{
    Z_\alpha=\tfrac{1}{\sqrt2}(e_\alpha-i\Phi e_\alpha),\qquad
    \bar Z_\alpha=\tfrac{1}{\sqrt2}(e_\alpha+i\Phi e_\alpha).
}

Now consider $u=h\circ\pi:\S^{2n+1}\to\CP^n$. Since $h$ is holomorphic, we have
\eq{\label{eq:h1}
    J\rd u(X)=\rd u(\Phi X),\qquad X\in\Gamma(T\S^{2n+1}),
}
so $u$ is transversely holomorphic. In particular,
\eq{
    J\rd u(Z_\alpha)=i\rd u(Z_\alpha),\qquad J\rd u(\bar Z_\alpha)=-i\rd u(\bar Z_\alpha),
}
and hence
\eq{
    \rd u(Z_\alpha)\in\Gamma(u^*T^{1,0}\CP^n),
    \qquad
    \rd u(\bar Z_\alpha)\in\Gamma(u^*T^{0,1}\CP^n).
}

\subsection{The complexified second variation}
For a variation field $w\in\Gamma(u^*T\CP^n)$, recall the index form
\eq{
    Q(w,w)=\int_{\S^{2n+1}}\Big(\abs{\nabla w}^2-\sum_{j=1}^{2n+1}R\big(w,\rd u(e_j),w,\rd u(e_j)\big)\Big),
}
where $\{e_j\}_{j=1}^{2n+1}$ is any local orthonormal frame on $\S^{2n+1}$ and $R=R^{\CP^n}$ is the $(0,4)$-tensor
$R(X,Y,Z,W)\coloneqq\<R(X,Y)W,Z\>$.

Complexify $u^*T\CP^n$ and extend $\nabla$ and $R$ complex-multilinearly. For $V\in\Gamma(u^*T^{1,0}\CP^n)$ define
\eq{\label{eq_complexified_second_variation}
    Q(V,\bar V)\coloneqq \int_{\S^{2n+1}}\Big(\<\nabla V,\nabla\bar V\>-\sum_{j=1}^{2n+1}R\big(V,\rd u(e_j),\bar V,\rd u(e_j)\big)\Big).
}

\begin{lemma}
For $V\in\Gamma(u^*T^{1,0}\CP^n)$ and $w\coloneqq V+\bar V$, one has $Q(w,w)=2Q(V,\bar V)$.
\end{lemma}

\begin{proof}
We use the complex-bilinear extension $\<\cdot,\cdot\>$ of the metric. For any $X\in\Gamma(T\S^{2n+1})$,
\eq{
    \abs{\nabla_X w}^2
    =\<\nabla_X(V+\bar V),\nabla_X(V+\bar V)\>
    =2\<\nabla_XV,\nabla_X\bar V\>,
}
since $\<\nabla_XV,\nabla_XV\>=\<\nabla_X\bar V,\nabla_X\bar V\>=0$. Summing over an orthonormal frame gives
\eq{\label{eq:-2}
    \abs{\nabla w}^2=2\<\nabla V,\nabla\bar V\>.
}

For the curvature term, choose an adapted orthonormal frame $\{e_1,\dots,e_{2n},e_{2n+1}=\xi\}$. By \eqref{eq_only_non-zero}, for each $\alpha\in\{1,\dots,n\}$,
\eq{
    &R\big(w,\rd u(e_\alpha),w,\rd u(e_\alpha)\big)+R\big(w,\rd u(\Phi e_\alpha),w,\rd u(\Phi e_\alpha)\big)\\
    &=R\big(w,\rd u(e_\alpha),w,\rd u(e_\alpha)\big)+R\big(Jw,\rd u(e_\alpha),Jw,\rd u(e_\alpha)\big)\\
    &=2\Big\{R\big(V,\rd u(e_\alpha),\bar V,\rd u(e_\alpha)\big)+R\big(V,\rd u(\Phi e_\alpha),\bar V,\rd u(\Phi e_\alpha)\big)\Big\},
}
where we used \eqref{eq:JV}, \eqref{eq_curvature_tensor_additional}, and \eqref{eq:h1}. Since $\rd u(\xi)=0$, we obtain
\eq{\label{eq:-3}
    \sum_{j=1}^{2n+1}R\big(w,\rd u(e_j),w,\rd u(e_j)\big)
    =2\sum_{j=1}^{2n+1}R\big(V,\rd u(e_j),\bar V,\rd u(e_j)\big).
}
Combining \eqref{eq:-2} and \eqref{eq:-3} yields $Q(w,w)=2Q(V,\bar V)$.
\end{proof}

We now introduce the following operators:
\eq{
    \bar{\p}_bV\coloneqq (\nabla_{\bar{Z}_1}V,\cdots,\nabla_{\bar{Z}_n}V), \qquad {\p}_bV\coloneqq (\nabla_{Z_1}V,\cdots,\nabla_{Z_n}V).
}
\begin{lemma}\label{lemC.4}
We have
\eq{
     Q(V,\bar{V}) = \int \Big( \abs{\bar{\p}_bV}^2 + \abs{{\p}_b V}^2 + \abs{\nabla_\xi V}^2 + \sum_{\alpha=1}^n R\big(V,\bar{V},\rd u(Z_\alpha),\rd u(\bar{Z}_\alpha)\big) \Big).
}
\end{lemma}
\begin{proof}
We start from the definition of the complexified index form and split the covariant derivatives into horizontal and vertical parts:
\eq{
    Q(V,\bar{V})
    &= \int \Big( \sum_{\alpha=1}^n \abs{\nabla_{Z_\alpha}V}^2 + \sum_{\alpha=1}^n \abs{\nabla_{\bar{Z}_\alpha}V}^2 + \abs{\nabla_\xi V}^2 - \sum_{j=1}^{2n+1} R\big(V,\rd u(e_j),\bar{V},\rd u(e_j)\big) \Big) \\
    &= \int \Big( \abs{\bar{\p}_bV}^2 + \abs{{\p}_bV}^2 + \abs{\nabla_\xi V}^2 - \sum_{j=1}^{2n+1} R\big(V,\rd u(e_j),\bar{V},\rd u(e_j)\big) \Big).
}
Since $u=h\circ\pi$ is constant along the Reeb direction, $\rd u(\xi)=0$, so only the horizontal directions contribute to the curvature term. Using \eqref{eq_curvature_tensor_additional} and \eqref{eq_only_non-zero}, we compute
\begin{align*}
    \sum_{j=1}^{2n+1} R\big(V,\rd u(e_j),\bar{V},\rd u(e_j)\big)
    &= \sum_{\alpha=1}^n R\big(V,\rd u(e_\alpha),\bar{V},\rd u(e_\alpha)\big)
     + \sum_{\alpha=1}^n R\big(V,\rd u(\Phi e_\alpha),\bar{V},\rd u(\Phi e_\alpha)\big) \\
    &= \sum_{\alpha=1}^n R\big(V,\rd u(e_\alpha),\bar{V},\rd u(e_\alpha)\big)
     + \sum_{\alpha=1}^n R\big(JV,\rd u(e_\alpha),J\bar{V},\rd u(e_\alpha)\big) \\
    &= 2\sum_{\alpha=1}^n R\big(V,\rd u(e_\alpha),\bar{V},\rd u(e_\alpha)\big) \\
    &= \sum_{\alpha=1}^n R\big(V,\rd u(Z_\alpha)+\rd u(\bar Z_\alpha),\bar V,\rd u(Z_\alpha)+\rd u(\bar Z_\alpha)\big) \\
    &= \sum_{\alpha=1}^n R\big(V,\rd u(\bar Z_\alpha),\bar V,\rd u(Z_\alpha)\big),
\end{align*}
where in the last step we used \eqref{eq_only_non-zero} to discard the pure-type terms.

Finally, applying the first Bianchi identity to $(V,\bar V,\rd u(Z_\alpha),\rd u(\bar Z_\alpha))$ and again using \eqref{eq_only_non-zero}, we obtain
\eq{
    R\big(V,\rd u(\bar Z_\alpha),\bar V,\rd u(Z_\alpha)\big) = -R\big(V,\bar V,\rd u(Z_\alpha),\rd u(\bar Z_\alpha)\big).
}
Substituting this identity back into the expression for $Q(V,\bar V)$ yields the claimed formula.
\end{proof}

\begin{lemma}\label{lem:C5}
We have
\eq{
\int \Big(\abs{\bar\p_b V}^2 - \abs{\p_b V}^2 -2ni\<\nabla_\xi V,\bar{V}\> \Big)
=\int \sum_{\alpha=1}^n  R\big(V,\bar{V},\rd u(Z_\alpha),\rd u(\bar{Z}_\alpha)\big).
}
\end{lemma}

\begin{proof}
We integrate by parts w.r.t. $\{Z_\alpha\}$. Since $\p_bV=(\nabla_{Z_\alpha}V)_\alpha$ and $M$ is closed,
\eq{\label{eq:int-by-parts-C5}
\int |\p_b V|^2
= -\int \sum_{\alpha=1}^n \big\langle \nabla_{\bar Z_\alpha}\nabla_{Z_\alpha}V,\bar V\big\rangle
-\int \sum_{\alpha=1}^n {\rm div}(\bar Z_\alpha)\,\big\langle \nabla_{Z_\alpha}V,\bar V\big\rangle.
}
Similarly, integrating by parts in the $\{\bar Z_\alpha\}$-directions gives
\eq{\label{eq:int-by-partsbar-C5}
\int |\bar\p_b V|^2
= -\int \sum_{\alpha=1}^n \big\langle \nabla_{Z_\alpha}\nabla_{\bar Z_\alpha}V,\bar V\big\rangle
-\int \sum_{\alpha=1}^n {\rm div}(Z_\alpha)\,\big\langle \nabla_{\bar Z_\alpha}V,\bar V\big\rangle.
}
Now take the difference of \eqref{eq:int-by-partsbar-C5} and \eqref{eq:int-by-parts-C5}. Using the commutator identity
$\nabla_{\bar Z_\alpha}\nabla_{Z_\alpha}V-\nabla_{Z_\alpha}\nabla_{\bar Z_\alpha}V=R(\bar Z_\alpha,Z_\alpha)V+\nabla_{[\bar Z_\alpha,Z_\alpha]}V$,
we eliminate the second-derivative terms and obtain
\eq{\label{eq2_lemC.5}
\int |\bar\p_b V|^2
- \int |\p_b V|^2
= \int \sum_{\alpha=1}^n \big\langle R(\bar Z_\alpha,Z_\alpha)V,\bar V\big\rangle
- \int \big\langle \nabla_W V,\bar V\big\rangle,
}
where
\eq{\label{eq:def-W-C5}
W\coloneqq \sum_{\alpha=1}^n [Z_\alpha,\bar Z_\alpha]
+\sum_{\alpha=1}^n {\rm div}(Z_\alpha)\,\bar Z_\alpha
-\sum_{\alpha=1}^n {\rm div}(\bar Z_\alpha)\,Z_\alpha.
}
The curvature term satisfies
$\langle R(\bar Z_\alpha,Z_\alpha)V,\bar V\rangle=R(V,\bar V,\rd u(Z_\alpha),\rd u(\bar Z_\alpha))$ by the definition of the curvature of $u^*TN$.

It remains to identify $W$.

\begin{lemma}\label{lem:divergence-identity-C5}
For any local unitary frame $\{Z_\alpha\}_{\alpha=1}^n$ of $T^{1,0}\mathcal{H}$, one has
\eq{\label{eq:divergence-identity-C5}
\sum_{\alpha=1}^n [Z_\alpha,\bar Z_\alpha]^h
=\sum_{\alpha=1}^n {\rm div}(\bar Z_\alpha)\,Z_\alpha
-\sum_{\alpha=1}^n {\rm div}(Z_\alpha)\,\bar Z_\alpha.
}
\end{lemma}

\begin{proof}
Since $\{Z_\alpha,\bar Z_\alpha\}$ is unitary, every horizontal vector field $X$ decomposes as
$X=\sum_{\beta}\langle X,\bar Z_\beta\rangle Z_\beta+\sum_{\beta}\langle X,Z_\beta\rangle\bar Z_\beta$.
Apply this to $X=\sum_{\alpha}[Z_\alpha,\bar Z_\alpha]^h$.
Using $[Z_\alpha,\bar Z_\alpha]=\nabla_{Z_\alpha}\bar Z_\alpha-\nabla_{\bar Z_\alpha}Z_\alpha$ and metric compatibility, we have
\eq{
    \langle \sum_{\alpha}[Z_\alpha,\bar Z_\alpha],\bar Z_\beta \rangle =-\langle \sum_{\alpha}\nabla_{\bar Z_\alpha}Z_\alpha,\bar Z_\beta \rangle, \qquad \langle \sum_{\alpha}[Z_\alpha,\bar Z_\alpha],Z_\beta \rangle = \langle \sum_{\alpha}\nabla_{Z_\alpha}\bar Z_\alpha,Z_\beta \rangle.
}
By the definition of ${\rm div}$ used above (cf. the computation after \eqref{eq2_lemC.5}), these coefficients equal ${\rm div}(\bar Z_\beta)$ and $-{\rm div}(Z_\beta)$, respectively, and \eqref{eq:divergence-identity-C5} follows.
\end{proof}

For the vertical part, with $\eta=\xi^\flat$ we have
$\eta([Z_\alpha,\bar Z_\alpha])=-\rd\eta(Z_\alpha,\bar Z_\alpha)=-2i$ on a Sasakian manifold, hence
$\eta(W)=\sum_{\alpha}\eta([Z_\alpha,\bar Z_\alpha])=-2ni$.
Therefore $W=-2ni\xi$, and \eqref{eq2_lemC.5} becomes
\eq{
\int \Big(|\bar\p_b V|^2-|\p_b V|^2-2ni\langle \nabla_\xi V,\bar V\rangle\Big)
=\int \sum_{\alpha=1}^n R\big(V,\bar V,\rd u(Z_\alpha),\rd u(\bar Z_\alpha)\big),
}
which is the desired identity.
\end{proof}
\begin{corollary}\label{lemC.7}
We have 
\eq{\label{eq:Q_V}
    Q(V,\bar{V}) = \int \Big( 2\abs{\bar{\p}_bV}^2 + \abs{\nabla_\xi V}^2 - 2ni\<\nabla_\xi V,\bar{V}\> \Big).
}
\end{corollary}

It is noteworthy that this expression contains no explicit curvature term.
In particular, if  $\nabla_\xi V=0$ (i.e. $V$ is basic), then $Q(V,\bar V)\ge 0$. Thus $u$ is stable under basic (horizontal) variations; equivalently, the holomorphic map $h:\CP^n\to\CP^n$ is stable (indeed, it minimizes energy in its homotopy class). Consequently, any unstable directions for $u$ must come from variations with a nontrivial vertical component.

\begin{corollary}\label{corC.5}
If $V\in{\rm ker}(\bar{\p}_b)$, then
\eq{
    Q(V,\bar{V}) = \int \Big( \abs{\nabla_\xi V}^2 - 2ni\<\nabla_\xi V,\bar{V}\> \Big). 
}
\end{corollary}

\subsection{Decomposition in terms of eigenspaces of \texorpdfstring{$-i\nabla_\xi$}{-i nabla xi}}

The following lemma is well known, see \cite{Lee88} for the scalar function case, and \cite{CHT19} for the bundle case. We give a proof here for the reader's convenience.

\begin{lemma}\label{lem:commute-xi-kohn}
Let $u:\S^{2n+1} \to \CP^n$ with $\rd u(\xi)=0$ (i.e. $u$ is constant along the Reeb direction). Let $\Box_b\coloneqq \bar\p_b^*\bar\p_b$ be the (bundle-valued) Kohn Laplacian associated with the pullback connection on $u^*T^{1,0}\CP^n$. Then
\eq{\label{eq:boxb-xi-commute}
    [\Box_b,\nabla_\xi]=0.
}
\end{lemma}

\begin{proof}
Fix $\bar Z\in T^{0,1}\mathcal{H}$. Using the definition $(\bar\p_bV)(\bar Z)=\nabla_{\bar Z}V$ and the general commutator identity for a connection,
\eq{
\nabla_\xi\nabla_{\bar Z}V-\nabla_{\bar Z}\nabla_\xi V
=\nabla_{[\xi,\bar Z]}V + R^{u^*T^{1,0}\CP^n}(\xi,\bar Z)V,
}
we obtain
\eq{\label{eq:xi-bar-pb-comm}
\big(\nabla_\xi(\bar\p_bV)-\bar\p_b(\nabla_\xi V)\big)(\bar Z)
=\nabla_{[\xi,\bar Z]-\nabla_\xi\bar Z}V + R^{u^*T^{1,0}\CP^n}(\xi,\bar Z)V.
}
Since the pullback curvature satisfies
$R^{u^*T^{1,0}\CP^n}(\xi,\bar Z)V=R^{\CP^n}(\rd u(\xi),\rd u(\bar Z))V=0$ by $\rd u(\xi)=0$, and since $\nabla_\xi\bar Z=[\xi,\bar Z]+\nabla_{\bar Z}\xi$ with $\nabla_{\bar Z}\xi=-i\bar Z$ on a Sasakian manifold, \eqref{eq:xi-bar-pb-comm} becomes
$$
\big(\nabla_\xi\bar\p_b-\bar\p_b\nabla_\xi\big)V=i\,\bar\p_bV.
$$
Taking $L^2$-adjoint yields $[\nabla_\xi,\bar\p_b^*]=-i\bar\p_b^*$, hence
$$
[\nabla_\xi,\Box_b]=[\nabla_\xi,\bar\p_b^*\bar\p_b]=[\nabla_\xi,\bar\p_b^*]\,\bar\p_b+\bar\p_b^*\,[\nabla_\xi,\bar\p_b]=0,
$$
which proves \eqref{eq:boxb-xi-commute}.
\end{proof}

Since $\Box_b$ commutes with $\nabla_\xi$, we may decompose $\Gamma(u^*T^{1,0}\CP^n)$ into eigenspaces of the self-adjoint operator $-i\nabla_\xi$ and write
\eq{
    Q(V,\bar{V}) = \sum_{m\in\mathbb{Z}} Q(V_m,\bar{V}_m),
}
where $V_m\in E_m$ satisfies
\eq{
    \nabla_\xi V_m = im\,V_m.
}
Hence it suffices to analyze $Q$ on each $E_m$ separately. If $V\in\ker(\bar\p_b)\cap E_m$, then Corollary~\ref{corC.5} gives
\eq{
    Q(V,\bar{V}) = \int m(m+2n)\abs{V}^2.
}
In particular, $Q$ is negative precisely for
\eq{
    -2n+1\leq m\leq -1.
}
This yields the following lower bounds.
\begin{proposition}\label{propC.6}
Under the assumption of Theorem~\ref{thmC.1}, we have
\eq{
    {\rm ind}(u) \ge \sum_{m=-(2n-1)}^{-1} {\rm dim}_{\R} ({\rm ker}(\bar{\p}_b)\cap E_m)
}
and
\eq{
    {\rm nul}(u) \ge {\rm dim}_{\R} ({\rm ker}(\bar{\p}_b)\cap E_0) + {\rm dim}_{\R} ({\rm ker}(\bar{\p}_b)\cap E_{-2n}).
}
\end{proposition}

\subsection{Dimension counting, Proof of Theorem \ref{thmC.1}}\label{sec4.5}
It remains to compute $\dim_\R(\ker(\bar\p_b)\cap E_m)$ for $m\in\{-2n,-2n+1,\dots,0\}$. Denote
\eq{
    \mathcal{P}_k\coloneqq\{\hbox{holomorphic homogeneous polynomials of algebraic degree $k$ on $\C^{n+1}$}\}.
}
In this subsection, all parameters of curves in projective space are complex, and derivatives of such curves denote their holomorphic differentials, taking values in the $(1,0)$-tangent bundle.

We start with a standard description of the tangent space of projective space
\eq{\label{eq_tan_CP^n}
    T_{[z]}^{1,0}\mathbb{CP}^n \cong \C^{n+1}/\C z.
}
Indeed, for $[z]\in\CP^n$ with $z\in\C^{n+1}\backslash\{0\}$ and $\mathbf{b}\in\C^{n+1}$, the curve $t\mapsto[z+t\mathbf{b}]$ is well defined for small $t$, and differentiating at $t=0$ produces a tangent vector at $[z]$:
\eq{
    [z+t\mathbf{b}]\in\mathbb{CP}^n.
}
Define
\eq{
    V_{\mathbf{b}} \coloneqq \frac{\rd}{\rd t}\Big|_{t=0} [z+t\mathbf{b}] \in T_{[z]}^{1,0}\mathbb{CP}^n.
}
If $\<\mathbf{b}\>$ denotes the class of $\mathbf{b}$ in $\C^{n+1}/\C z$, then $V_{\mathbf{b}}$ depends only on $\<\mathbf{b}\>$: for any $\lambda\in\C$,
\eq{
    V_{\mathbf{b}+\lambda z}
    &= \frac{\rd}{\rd t}\Big|_{t=0} [z+t(\mathbf{b}+\lambda z)]
     = \frac{\rd}{\rd t}\Big|_{t=0} [(1+\lambda t)z+t\mathbf{b}] \\
    &= \frac{\rd}{\rd t}\Big|_{t=0} \Big[z+\frac{t}{1+\lambda t}\mathbf{b}\Big]
     = \frac{\rd}{\rd s}\Big|_{s=0} [z+s\mathbf{b}] \cdot \frac{\rd}{\rd t}\Big|_{t=0} \frac{t}{1+\lambda t}
     = V_{\mathbf{b}}.
}
Thus we obtain the identification
\eq{
    I_z : \C^{n+1}/\C z \overset{\cong}{\to} T_{[z]}^{1,0}\CP^n, \quad \<\mathbf{b}\> \mapsto V_{\mathbf{b}}.
}

Now assume $u=h\circ\pi$ as in Theorem~\ref{thmC.1}, with algebraic degree $d$. Writing, for $z\in\S^{2n+1}\subset\C^{n+1}$,
\eq{
    u(z)=h([z])=[p_0(z):\cdots:p_n(z)]\eqcolon[\mathbf{p}(z)], \quad p_\alpha\in\mathcal{P}_d
}
Therefore, given any $\mathbf{b}\in C^\infty(\S^{2n+1},\C^{n+1})$, we have
\eq{
    I_{\mathbf{p}(z)} : \C^{n+1}/\C \mathbf{p}(z) \cong T_{\mathbf{p}(z)]}^{1,0}\CP^n = T_{u(z)}^{1,0}\CP^n, \quad \<\mathbf{b}(z)\> \mapsto V_{\mathbf{b}}(z),
}
where $\mathbf{b}(z)\sim \mathbf{b}(z)+\lambda(z)\mathbf{p}(z)$. In particular,
\eq{
    V_{\mathbf{b}}(z)\coloneqq \frac{\rd}{\rd t}\Big|_{t=0} [\mathbf{p}(z)+t\mathbf{b}(z)] = I_{\mathbf{p}(z)}\big(\<\mathbf{b}(z)\>\big) \in \Gamma(u^*T^{1,0}\CP^n)
}
is an admissible complex variation field with associated real variation field $V_{\mathbf{b}}+\bar{V}_{\mathbf{b}}$. Since we only consider $u=h\circ\pi$ in this section, we omit the subscript $h$ in the notation.

\begin{remark}
The field $V_{\mathbf{b}}$ is smooth. Indeed, in a chart $(w_1,\dots,w_n)$ on $\CP^n$ (for example $w_\alpha(z)=z_\alpha/z_0$ on $\{z_0\neq 0\}$),
\eq{
    w_\alpha([\mathbf{p}(z)+t\mathbf{b}(z)]) = \frac{p_\alpha(z)+tb_\alpha(z)}{p_0(z)+tb_0(z)}
}
and
\eq{\label{eq_holo_derivative}
    \frac{\rd}{\rd t}\Big|_{t=0}w_\alpha([\mathbf{p}(z)+t\mathbf{b}(z)])
    = \frac{p_0(z)b_\alpha(z)-p_\alpha(z)b_0(z)}{p_0(z)^2}\in C^\infty.
}
Hence
\eq{
    V_{\mathbf{b}}(z) = \frac{\rd}{\rd t}\Big|_{t=0} [\mathbf{p}(z)+t\mathbf{b}(z)]
    = \frac{\rd}{\rd t}\Big|_{t=0}w_\alpha([\mathbf{p}(z)+t\mathbf{b}(z)])\,\frac{\p}{\p w_\alpha}\Big|_{u(z)}
}
is smooth.
\end{remark}

\begin{lemma}\label{lem_V_B}
If $\mathbf{b}\in(\mathcal{P}_k)^{n+1}$, then $V_{\mathbf{b}}\in {\rm ker}(\bar{\p}_b)\cap E_{k-d}$.
\end{lemma}
\begin{proof}
The inclusion $V_{\mathbf{b}}\in\ker(\bar\p_b)$ follows immediately from \eqref{eq_holo_derivative}. It remains to show $V_{\mathbf{b}}\in E_{k-d}$. Let $\theta\mapsto e^{i\theta}z$ be the Hopf $\S^1$-action. Homogeneity gives
\eq{
    \mathbf{p}(e^{i\theta}z) = e^{id\theta}\mathbf{p}(z), \qquad \mathbf{b}(e^{i\theta}z) = e^{ik\theta}\mathbf{b}(z).
}
Therefore,
\eq{
    V_{\mathbf{b}}(e^{i\theta}z)
    &= \frac{\rd}{\rd t}\Big|_{t=0} [\mathbf{p}(e^{i\theta}z)+t\mathbf{b}(e^{i\theta}z)]
     = \frac{\rd}{\rd t}\Big|_{t=0} [e^{id\theta}\mathbf{p}(z)+te^{ik\theta}\mathbf{b}(z)] \\
    &= \frac{\rd}{\rd t}\Big|_{t=0} [\mathbf{p}(z)+te^{i(k-d)\theta}\mathbf{b}(z)]
     = e^{i(k-d)\theta}V_{\mathbf{b}}(z),
}
so $V_{\mathbf{b}}\in E_{k-d}$.
\end{proof}

Define the linear map $T:(\mathcal{P}_k)^{n+1}\to\Gamma(u^*T^{1,0}\CP^n)$ 
by $T(\mathbf{b})\coloneqq V_{\mathbf{b}}$.

\begin{lemma}\label{lemC.9}
(1) If $k<d$, then $T$ is injective.

\noindent(2) If $k\ge d$, then
\eq{
    \ker(T)=\{\lambda \mathbf{p}\mid \lambda\in\mathcal{P}_{k-d}\}.
}
\end{lemma}
\begin{proof}
Suppose $T(\mathbf{b})=V_{\mathbf{b}}\equiv 0$. Then $\<\mathbf{b}(z)\>=0$ for all $z$, i.e.
\eq{
    \mathbf{b}(z) = \lambda(z)\mathbf{p}(z)
}
for some smooth function $\lambda(z)$.  
Then $\lambda=b_\alpha/p_\alpha$ is rational for each $0\le\alpha\le n$. Write $\lambda=R/S$ with $R, S$ coprime.  Then $b_\alpha/p_\alpha=R/S$ implies
\eq{
 b_\alpha S = p_\alpha R, \qquad 0\le\alpha\le n.
}
Since $R$ and $S$ are coprime, it follows that $S$ divides each $p_\alpha$. Since $p_0,\dots,p_n$ are coprime, we conclude that $S$ is constant. Since $\mathbf{b}\in(\mathcal{P}_k)^{n+1}$ and $\mathbf{p}\in(\mathcal{P}_d)^{n+1}$, we have $\lambda\in\mathcal{P}_{k-d}$. The claim then follows.
\end{proof}

Finally, we prove Theorem~\ref{thmC.1}.
\begin{proof}[Proof of Theorem~\ref{thmC.1}]
Recall that
\eq{
    {\rm dim}_{\C}\mathcal{P}_k = \binom{n+k}{n}.
}
Lemma~\ref{lemC.9} then gives that for $m=k-d$,
\eq{
    {\rm dim}_{\C}\big({\rm ker}(\bar{\p}_b)\cap E_m\big) \geq
    \begin{cases}
        (n+1)\binom{n+d+m}{n} - \binom{n+m}{n}, &\quad m\ge 0, \\
        (n+1)\binom{n+d+m}{n}, &\quad -d\le m<0.
    \end{cases}
}
Together with Proposition~\ref{propC.6}, this yields the claimed lower bounds.
\end{proof}

\begin{remark}
In the special case $\mathbf{p}(z)=z$, taking $\mathbf{b}=\frac{\p \mathbf{p}}{\p z_\beta}$ for a fixed $\beta$ yields
\eq{
    V_{\mathbf{b}} = \frac{\rd}{\rd t}\Big|_{t=0} \Big[\mathbf{p}+t\frac{\p \mathbf{p}}{\p z_\beta}\Big]
    = \frac{1}{2}\rd h\circ\rd\pi(\nabla \bar{z}_\beta)
    = \frac{1}{2}\rd u(\nabla \bar{z}_\beta),
}
which recovers the variations coming from conformal Killing vector fields on $\S^{2n+1}$.
\end{remark}

\section{Proof of Theorem~\ref{thm_nullity}: index part}\label{sec5}

In this section, we prove the following result, which completes the proof of the first statement in Theorem~\ref{thm_nullity}.

\begin{theorem}\label{thm5.1}
For any degree-one holomorphic map $h:\CP^n\to\CP^n$, and $u=h\circ\pi$, we have
\eq{
    {\rm ind}(u)=2n+2.
}
\end{theorem}

Since $h$ has degree one, there exists $A\in{\rm GL}(n+1,\C)$ such that
\eq{
    h([z]) = [Az].
}

Recall \eqref{eq:Q_V} that for any $V\in\Gamma(u^*T^{1,0}\CP^n)$, we have
\eq{\label{eq:9}
    Q(V,\bar{V}) = \int_{\S^{2n+1}} \Big( 2\abs{\bar{\p}_bV}^2 + \abs{\nabla_\xi V}^2 - 2ni\<\nabla_\xi V,\bar{V}\> \Big).
}

\subsection{Formulation using the tautological line bundle}

We say that $V$ has weight $m$, if
\eq{
    \nabla_\xi V = imV,
}
or equivalently
\eq{
    V(e^{i\theta}z) = e^{im\theta}V(z).
}

Let $\mathcal{O}(-1)$ be the tautological line bundle over $\CP^n$, which has fibre $\C z$ at each point $[z]\in\CP^n$. Set $\mathcal{O}(-m) \coloneqq \mathcal{O}(-1)^{\otimes m}$, defined by
\eq{
    \mathcal{O}(-m)_{[z]} = (\C z)^{\otimes m} = \C z\otimes \cdots\otimes\C z \quad\hbox{($m$ times)}.
}
For $m<0$, tensor powers are interpreted using the dual bundle, and $\mathcal{O}(0)$ is the trivial line bundle.

\begin{lemma}
If $V$ has weight $-m$, then $V\otimes z^{\otimes m}$ has weight $0$, i.e. is basic.
\end{lemma}
\begin{proof}
By definition we have
\eq{
    V(e^{i\theta}z) = e^{-im\theta}V(z),
}
and
\eq{
    (e^{i\theta}z)^{\otimes m} = e^{im\theta}z^{\otimes m}.
}
Hence
\eq{
    V(e^{i\theta}z)\otimes(e^{i\theta}z)^{\otimes m} = V(z)\otimes z^m, 
}
and the claim follows.
\end{proof}

As a consequence, $V\otimes z^{\otimes m}$ descends to $\CP^n$ as a section $s$ of the bundle
\eq{
    \mathcal{E}(m)\coloneqq h^*T^{1,0}\CP^n\otimes\mathcal{O}(-m),
}
which has fibre
\eq{
    \mathcal{E}(m)_{[z]}= T^{1,0}_{h([z])}\CP^n\otimes(\C z)^{\otimes m}.
}

\begin{lemma}\label{lemD.3}
If $V$ has weight $-m$, then the correspondence
\eq{\label{eq:10}
    V\ \mapsto\ s\in\Gamma(\mathcal{E}(m)) \quad\hbox{with}\quad (\pi^*s)(z) = V(z)\otimes z^{\otimes m} 
}
is pointwise isometric. Moreover, it induces a pointwise isometry
\eq{
    \bar{\p}_bV \quad\mapsto\quad \bar{\p}s.
}
\end{lemma}
\begin{proof}
Since $\abs{z}=1$ for $z\in\S^{2n+1}$, we have
\eq{
    \abs{z^{\otimes m}}=1.
}
Hence
\eq{
    \abs{s([z])} = \abs{V(z)\otimes z^{\otimes m}} = \abs{V(z)},
}
and the first claim follows. Moreover, if $X\in T_z\S^{2n+1}$ is a horizontal vector, then
\eq{
    \pi^*\nabla^{\mathcal{E}(m)}_{\rd\pi(X)}s = \nabla^{\pi^*\mathcal{E}(m)}_{X}(\pi^*s) = \nabla^{u^*T^{1,0}\CP^n}_XV\otimes z^{\otimes m}.
}
For each horizontal vector $X$, we have
\eq{
    \nabla^{\pi^*\mathcal{O}(-1)}_Xz = {\rm proj}_{\C z}X = 0.
}
Hence
\eq{
    \bar{\p}s(\rd\pi(X)) = \bar{\p}_bV(X)\otimes z^{\otimes m},
}
and the claim follows.
\end{proof}

\begin{lemma}
If $V$ has weight $-m$, and $s$ is the associated section via \eqref{eq:10}, then \eqref{eq:9} descends to
\eq{\label{eq:10.1}
    Q_m(s) \coloneqq \frac{1}{2\pi}Q(V,\bar{V}) = 2\int_{\CP^n} \abs{\bar{\p}s}^2 - m(2n-m)\int_{\CP^n}\abs{s}^2, \quad s\in\Gamma(\mathcal{E}(m)).
}
\end{lemma}
\begin{proof}
Since $V$ has weight $-m$, we have $\nabla_\xi V=-imV$. Then \eqref{eq:9} implies that
\eq{
    Q(V,\bar{V}) = \int_{\S^{2n+1}} \Big( 2\abs{\bar{\p}_bV}^2 + m^2\abs{V}^2 - 2mn\abs{V}^2 \Big).
}
Since each Hopf fibre is $\S^1$ with length $2\pi$, we have by Lemma~\ref{lemD.3} that
\eq{
    \int_{\S^{2n+1}} \abs{V}^2 = 2\pi\int_{\CP^n}\abs{s}^2,
}
and
\eq{
    \int_{\S^{2n+1}} \abs{\bar{\p}_bV}^2 = 2\pi\int_{\CP^n}\abs{\bar{\p}s}^2.
}
The claim follows.
\end{proof}

\subsection{Fix the representative}

\begin{lemma}\label{lemI.1}
Given $0\neq z\in\C^{n+1}$, we have
\eq{
    T^{1,0}_{[z]}\CP^n\cong{\rm Hom}_\C(\C z,\C^{n+1}/\C z).
}
\end{lemma}
\begin{proof}
Choose a curve in $\CP^n$, i.e. a family of lines $L(t)$ in $\C^{n+1}$, with $L(0)=L\in\CP^n$. Then
\eq{
    v \coloneqq L'(0) \in T^{1,0}_L\CP^n.
}
Given $z\in L$, choose a curve $z(t)$ in $\C^{n+1}$ such that $z(0)=z$ and $z(t)\in L(t)$.
Define a linear operator
\eq{
    \varphi_v: \C z \to \C^{n+1}/\C z, \quad z \mapsto \<z'(0)\> = z'(0) \!\!\mod \C z.
}
It is well-defined, i.e. independent of choice of $z(t)$. In fact, if $w(t)\in L(t)$ is another choice with $w(0)=z$, then
\eq{
    w(t) = \lambda(t)z(t)
}
for some function $\lambda(t)$ with $\lambda(0)=1$. Hence
\eq{
    w'(0)=\lambda'(0)z+z'(0) \equiv z'(0) \!\!\mod\C z.
}
Therefore, we obtain a linear operator
\eq{
    \varphi: T^{1,0}_{[z]}\CP^n\to{\rm Hom}_\C(\C z,\C^{n+1}/\C z), \quad v\mapsto \varphi_v.
}

If $\varphi_v=0$, then for any curve $z(t)$ with $z(0)=z$, we have
\eq{
    z'(0)=\lambda z
}
for some $\lambda\in\C$. We may choose
\eq{
    z(t) = z+tz'(0) = (1+\lambda t)z,
}
hence $L(t)=L=[z]$ and $v=L'(0)=0$. Therefore, $\varphi$ is injective.

Reversely, given a linear operator $T:\C z\to\C^{n+1}/\C z$, we have
\eq{
    T(z) = \<a\> = a \!\!\mod\C z
}
for some $a\in\C^{n+1}$. Define
\eq{
    L(t) \coloneqq [z+ta]\in\CP^n
}
and
\eq{
    v \coloneqq L'(0)\in T^{1,0}_{[z]}\CP^n.
}
We claim that $T=\varphi_v$. In fact, for $z(t)\coloneqq z+ta$, we have
\eq{
    \varphi_v(z) = \<z'(0)\> = \<a\> = T(z).
}
Hence $\varphi$ is surjective, and the conclusion follows.
\end{proof}

Recall the short exact (Euler) sequence
\eq{\label{eq:Euler}
    0 \to \mathcal{O}(-1) \to \underline{\C}^{n+1} \to \underline{\C}^{n+1}/\mathcal{O}(-1) \to 0,
}
where $\underline{\C}^{n+1}$ is the trivial bundle over $\CP^n$. Lemma~\ref{lemI.1} implies that
\eq{
    T^{1,0}\CP^n \cong {\rm Hom}(\mathcal{O}(-1),\underline{\C}^{n+1}/\mathcal{O}(-1)) = \underline{\C}^{n+1}/\mathcal{O}(-1) \otimes \mathcal{O}(1),
}
where we used the fact that $\mathcal{O}(1)$ is the dual bundle of $\mathcal{O}(-1)$. Tensoring \eqref{eq:Euler} with $\mathcal{O}(1-m)$ gives
\eq{\label{eq:11}
    0 \to \mathcal{O}(-m) \to \underline{\C}^{n+1}\otimes\mathcal{O}(1-m) \to T^{1,0}\CP^n\otimes\mathcal{O}(-m) \to 0.
}

In order to fix the choice of representative, we consider
\eq{
    h([z])=[Az]=[p]\in\CP^n, \quad p\coloneqq Az\in\C^{n+1}.
}
Since
\eq{
    \C z \mapsto \C Az
}
is a linear isomorphism (not necessarily isometric), we have
\eq{
    h^*\mathcal{O}(-m) \cong \mathcal{O}(-m).
}
Hence the pull-back by $h$ of \eqref{eq:11} gives
\eq{\label{eq:12}
    0 \to \mathcal{O}(-m) \to \underline{\C}^{n+1}\otimes\mathcal{O}(1-m) \overset{q}{\to} \mathcal{E}(m) \to 0,
}
where $q$ is given by $q(\mathbf{b})=\mathbf{b}\!\!\mod\C p$. For any $s\in\Gamma(\mathcal{E}(m))$, 
the pre-image is
\eq{
    \mathbf{b} \in \Gamma(\underline{\C}^{n+1}\otimes\mathcal{O}(1-m)) \mod \C p,
}
namely $s=q(\mathbf{b})$. Hence
\eq{
    W \coloneqq p\w\mathbf{b}\in\Gamma(\Lambda^2\C^{n+1}\otimes\mathcal{O}(2-m))
}
is independent of the choice of representatives.

\begin{lemma}\label{lemD.6}
Set $\rho\coloneqq\abs{p}^2$. Then
\eq{
    \abs{s}^2 = \rho^{-2}\abs{W}^2, \qquad \abs{\bar{\p}s}^2=\rho^{-2}\abs{\bar{\p}W}^2.
}
\end{lemma}
\begin{proof}
We decompose
\eq{
    \mathbf{b} = \mathbf{b}^\perp + \lambda p, \qquad \mathbf{b}^\perp\perp p.
}
It follows that
\eq{\label{eq:13}
    \abs{s}^2 = \abs{q(\mathbf{b})}^2 = \frac{\abs{\mathbf{b}^\perp}^2}{\abs{p}^2} = \rho^{-1}\abs{\mathbf{b}^\perp}^2.
}
Since $W=p\w\mathbf{b}=p\w\mathbf{b}^\perp$, we have
\eq{\label{eq:14}
    \abs{W}^2 = \abs{p}^2\abs{\mathbf{b}^\perp}^2 = \rho\abs{\mathbf{b}^\perp}^2. 
}
Combining \eqref{eq:13} and \eqref{eq:14} gives the first claim.

Next, for any $\mathbf{b}'=\mathbf{b}+fp$ we have
\eq{
    \bar{\p}\mathbf{b}' = \bar{\p}\mathbf{b} + (\bar{\p}f)p \equiv \bar{\p}\mathbf{b} \!\!\mod \C p.
}
Hence $\bar{\p}\mathbf{b}$ is a pre-image of $\bar{\p}s$, namely
\eq{
    \bar{\p}s = q(\bar{\p}\mathbf{b}).
}
Moreover, we have
\eq{
    \bar{\p}W = \bar{\p}(p\w\mathbf{b}) = p\w\bar{\p}\mathbf{b}. 
}
Hence the second claim follows similarly.
\end{proof}

\subsection{The spectral gap for \texorpdfstring{$m\leq0$}{m<=0} and \texorpdfstring{$m\geq2$}{m>=2} }

This case $m\le 0$ is trivial, since \eqref{eq:10.1} implies that
\eq{
    Q_m(s) \geq 2\int_{\CP^n}\abs{\bar{\p}s}^2 \geq 0.
}

We now consider the case $m\ge 2$. 
\begin{proposition}
For $m\geq2$ we have
\eq{\label{eq:14.9}
    \int_{\CP^n}\abs{\bar{\p}s}^2 \geq 2(n-1)(m-1)\int_{\CP^n}\abs{s}^2.
}
\end{proposition}
\begin{proof}
It is clear that if $n=1$, then the statement is true. Hence we  may assume $n\ge 2.$

The idea of proof  is to restrict the problem to hyperplanes by dual averaging, and then apply the Bochner--Kodaira identity to obtain the spectral gap.

Let $\rd\sigma$ be a $U(n+1)$-invariant probability measure on the dual unit sphere $S(\C^{n+1})^*$. Set
\eq{
    I \coloneqq \int_{S(\C^{n+1})^*} \delta_\C(\ell(p))\abs{\ell(\mathbf{b})}^2 \rd\sigma(\ell),
}
where $\ell\in S(\C^{n+1})^*$ and $\delta_\C$ is the Dirac measure on $\C$ centered at $0$. Note that
\eq{\label{eq:15}
    \delta_\C(\ell(p)) = \abs{p}^{-2}\delta_\C(\ell(p/\abs{p})) = \rho^{-1}\delta_\C(\ell(p/\abs{p})).
}
On the support of $\delta_\C(\ell(p/\abs{p}))$, which is the hyperplane $\{\ell(p)=0\}$, we have
\eq{\label{eq:16}
    \ell(\mathbf{b}) = \ell(\mathbf{b}^\perp).
}
Denote
\eq{
    S(p^\perp)^* \coloneqq S(\C^{n+1})^*\cap\{\ell\mid\ell(p)=0\}.
}
Using \eqref{eq:15} and \eqref{eq:16} we have
\eq{\label{eq:17}
    I = c_0(n)\rho^{-1}\int_{S(p^\perp)^*} \abs{\ell(\mathbf{b}^\perp)}^2 \rd\sigma_{S(p^\perp)^*}(\ell).
}
for some constant $c_0(n)>0$. It is clear that the right-hand side of \eqref{eq:17} is quadratic in $\mathbf{b}^\perp$ and depends only on the length of $\mathbf{b}^\perp$. Hence
\eq{
    I = c_1(n)\rho^{-1}\abs{\mathbf{b}^\perp}^2 = c_1(n)\abs{s}^2.
}
for some constant $c_1(n)>0$, where we used \eqref{eq:13}. The Fubini Theorem then gives
\eq{\label{eq:18}
    \int_{\CP^n}\abs{s}^2 = c_1(n)^{-1}\int_{S(\C^{n+1})^*} \Big(\int_{\CP^n} \delta_\C(\ell(p))\abs{\ell(\mathbf{b})}^2 \rdV_{\CP^n} \Big) \rd\sigma(\ell).
}
In order to compute the interior integral, we note that
\eq{
    \ell(p) = \ell(Az) = A^*\ell(z).
}
Since $A$ is invertible, in $S(\C^{n+1})^*$ we may write
\eq{
    A^*\ell = \abs{A^*\ell}\,\ell_0, \qquad \abs{\ell_0}=1.
}
It follows that
\eq{
    \delta_\C(\ell(p)) = \delta_\C(A^*\ell(z)) = \abs{A^*\ell}^{-2}\delta_\C(\ell_0(z)).
}
Hence
\eq{
    \int_{\CP^n} \delta_\C(\ell(p))\abs{\ell(\mathbf{b})}^2 \rdV_{\CP^n} = c_2(n)\int_{H_\ell} \abs{A^*\ell}^{-2}\abs{\ell(\mathbf{b})}^2 \rdV_{H_\ell},
}
where $H_\ell\coloneqq\{[z]\mid\ell_0(z)=0\}=\{[z]\mid\ell(p)=0\}$.
Inserting into \eqref{eq:18} gives
\eq{\label{eq:19}
    \int_{\CP^n}\abs{s}^2 = C(n)\int_{S(\C^{n+1})^*} \abs{A^*\ell}^{-2}\Big(\int_{H_\ell}\abs{\ell(\mathbf{b})}^2 \rdV_{H_\ell} \Big) \rd\sigma(\ell),
}
where $C(n)\coloneqq c_1(n)^{-1}c_2(n)$.

Similarly, applying Lemma~\ref{lemD.6} component-wise, we have
\eq{
    \int_{\CP^n}\abs{\bar{\p}s}^2 = C(n)\int_{S(\C^{n+1})^*} \abs{A^*\ell}^{-2}\Big(\int_{H_\ell}\abs{\ell(\bar{\p}\mathbf{b})}^2 \rdV_{H_\ell} \Big) \rd\sigma(\ell).
}
Note that the constant $C(n)$ is the same in both \eqref{eq:19} and \eqref{eq:19.1}. Since
\eq{
    \bar{\p}^{H_\ell}\ell(\mathbf{b}) = \ell(\bar{\p}\mathbf{b})^{H_\ell},
}
we have
\eq{\label{eq:42}
    \abs{\ell(\bar{\p}\mathbf{b})} \geq \abs{\bar{\p}^{H_\ell}(\ell(\mathbf{b}))}.
}
Hence
\eq{\label{eq:19.1}
    \int_{\CP^n}\abs{\bar{\p}s}^2 \geq C(n)\int_{S(\C^{n+1})^*} \abs{A^*\ell}^{-2}\Big( \int_{H_\ell}\abs{\bar{\p}^{H_\ell}\ell(\mathbf{b})}^2 \rdV_{H_\ell} \Big) \rd\sigma(\ell).
}

We now apply the Bochner--Kodaira--Nakano identity (see e.g. \cite{Demailly07}*{Chapter 7, Corollary 1.3})
\eq{\label{eq:20}
    \bar{\p}\bar{\p}^*+\bar{\p}^*\bar{\p} = \p\p^*+\p^*\p + [i\Theta(E),\Lambda],
}
where $E$ is a Hermitian holomorphic bundle, $\Theta(E)$ is the Chern curvature of $E$, and $\Lambda=L^*$, where $L=\omega\w$ is the Lefschetz operator.
In our case, we apply for the bundle $E = \mathcal{O}(1-m)$, and for the $\mathcal{O}(1-m)$-valued $0$-form $f\coloneqq\ell(\mathbf{b})$ on the hyperplane $H_\ell$. Since
\eq{
    i\Theta(\mathcal{O}(1)) = 2\omega,
}
where $\omega$ is the K\"ahler form on $H_\ell\cong\CP^{n-1}$, we have
\eq{
    i\Theta(\mathcal{O}(1-m)) = -2(m-1)\omega,
}
and
\eq{
    [i\Theta(\mathcal{O}(1-m)),\Lambda]f = -2(m-1)[\omega,\Lambda]f = 2(m-1)\abs{\omega}^2f = 2(m-1)(n-1)f.
}
Therefore, \eqref{eq:20} implies that
\eq{\label{eq:21}
    \int_{H_\ell}\abs{\bar{\p}^{H_\ell}\ell(\mathbf{b})}^2 \rdV_{H_\ell} &= \int_{H_\ell}\abs{\p^{H_\ell}\ell(\mathbf{b})}^2 \rdV_{H_\ell} + 2(n-1)(m-1)\int_{H_\ell}\abs{\ell(\mathbf{b})}^2 \rdV_{H_\ell} \\
    &\geq 2(n-1)(m-1)\int_{H_\ell}\abs{\ell(\mathbf{b})}^2 \rdV_{H_\ell}.
}
Combining \eqref{eq:19}, \eqref{eq:19.1}, and \eqref{eq:21} gives the desired estimate \eqref{eq:14.9}.    
\end{proof}

From \eqref{eq:10.1} and \eqref{eq:14.9}, we see that for $m\geq2$,
\eq{\label{eq:21.1}
    Q_m(s) \geq \{ 4(n-1)(m-1) - m(2n-m) \} \int_{\CP^n}\abs{s}^2 = (m-2)(m+2n-2)\int_{\CP^n}\abs{s}^2 \geq 0.
}

\subsection{The spectral gap for \texorpdfstring{$m=1$}{m=1}} \label{subsection5.5}

In this case, the quadratic form \eqref{eq:10.1} becomes
\eq{\label{eq:22}
    Q_1(s) = 2\int_{\CP^n} \abs{\bar{\p}s}^2 - (2n-1)\int_{\CP^n}\abs{s}^2, \quad s\in\Gamma(\mathcal{E}(1)).
}

First, the $m=1$ case clearly provides index at least $2n+2$. In fact, \eqref{eq:12} now becomes
\eq{
    0 \to \mathcal{O}(-1) \to \underline{\C}^{n+1} \to \mathcal{E}(1) \to 0.
}
It induces a long exact sequence of cohomology
\eq{
    0 \to H^0(\mathcal{O}(-1)) \to H^0(\underline{\C}^{n+1}) \to H^0(\mathcal{E}(1)) \to H^1(\mathcal{O}(-1)) \to \cdots.
}
Since
\eq{
    H^0(\mathcal{O}(-1)) = H^1(\mathcal{O}(-1)) = 0,
}
It follows that
\eq{
    {\rm dim}_\C H^0(\mathcal{E}(1)) = {\rm dim}_\C H^0(\underline{\C}^{n+1}) = n+1.
}
Therefore, each non-zero section $s\in\Gamma(H^0(\mathcal{E}(1)))$ gives
\eq{\label{eq:23}
    Q_1(s) = - (2n-1)\int_{\CP^n}\abs{s}^2 < 0.
}
Hence the $m=1$ case clearly provides index at least $2n+2$.

We now prove that there is no more index, which is a direct consequence of the following spectral gap.

\begin{proposition}\label{propD.8}
If $s\perp H^0(\mathcal{E}(1))$, then
\eq{
    \int_{\CP^n}\abs{\bar{\p}s}^2 \geq 2n\int_{\CP^n}\abs{s}^2.
}
\end{proposition}

Consider the operator
\eq{
    \bar{\p}:\Gamma(\mathcal{E}(1)) \to \Omega^{0,1}(\mathcal{E}(1)).
}
In order to obtain a good enough constant, we apply \eqref{eq:20} to a $(n,1)$-type form $\alpha$. Note that the Euler sequence gives
\eq{\label{eq:37}
    \Omega^{n,0}T^{1,0}\CP^n \cong {\rm det}T^{1,0}\CP^n \cong \mathcal{O}(n+1),
}
hence
\eq{
    \Omega^{n,0}(T^{1,0})^*\CP^n \cong \mathcal{O}(-(n+1)),
}
and consequently 
\eq{\label{eq:36}
    \Omega^{0,1}(\mathcal{E}(1)) &\cong \Gamma(\Omega^{0,1}(T^{1,0})^*\CP^n \otimes \mathcal{E}(1)) \\
    &\cong \Gamma(\Omega^{n,1}(T^{1,0})^*\CP^n \otimes\mathcal{O}(n+1) \otimes \mathcal{E}(1)) \\
    &\cong \Omega^{n,1}(h^*T^{1,0}\CP^n\otimes\mathcal{O}(n)).
}
Therefore, we choose
\eq{
    E = h^*T^{1,0}\CP^n\otimes\mathcal{O}(n), \quad \alpha\in\Omega^{n,1}(h^*T^{1,0}\CP^n\otimes\mathcal{O}(n)).
}

\begin{lemma}\label{lemJ.12}
The tangent bundle $T^{1,0}\CP^n$ is Nakano semi-positive, namely
\eq{
    i\Theta(T^{1,0}\CP^n)\geq0. 
}
\end{lemma}
\begin{proof}
Let $\{e_\alpha\}$ be a local frame of $E$, and denote the Chern curvature tensor by $R_{i\bar{j}\alpha\bar{\beta}}$, then
\eq{
    i\Theta(E) = \sum_{i,j,\alpha,\beta}R_{i\bar{j}\alpha\bar{\beta}}\rd z_i\w\rd\bar{z}_j\otimes e_\alpha^*\otimes e_\beta.
}
We need to prove that
\eq{
    \sum_{i,j,\alpha,\beta}R_{i\bar{j}\alpha\bar{\beta}}\xi_{i\alpha}\bar{\xi}_{j\beta} \geq 0
}
for any complex matrix $\xi=(\xi_{i\alpha})$. The curvature of $\CP^n$ w.r.t. the Fubini-Study metric is
\eq{
    R_{i\bar{j}\alpha\bar{\beta}} = 2( \delta_{ij}\delta_{\alpha\beta}+\delta_{i\beta}\delta_{j\alpha} ), 
}
hence
\eq{
    \sum_{i,j,\alpha,\beta}R_{i\bar{j}\alpha\bar{\beta}}\xi_{i\alpha}\bar{\xi}_{j\beta} = 2\sum_{i,\alpha}\abs{\xi_{i\alpha}}^2 + 2\sum_{i,\alpha}\xi_{i\alpha}\bar{\xi}_{\alpha i} = \sum_{i,\alpha}\abs{\xi_{i\alpha}+\xi_{\alpha i}}^2 \geq 0.
}
The conclusion follows.
\end{proof}

\begin{lemma}\label{lemJ.13}
The pull-back bundle $h^*T^{1,0}\CP^n$ is also Nakano semi-positive, namely
\eq{
    i\Theta(h^*T^{1,0}\CP^n)\geq0.
}
\end{lemma}
\begin{proof}
For any complex matrix $\xi=(\xi_{i\alpha})$, we have
\eq{\label{eq:34}
    \sum_{i,j,\alpha,\beta}R^{h^*T^{1,0}\CP^n}_{i\bar{j}\alpha\bar{\beta}}\xi_{i\alpha}\bar{\xi}_{j\beta} = \sum_{\mu,\nu,\alpha,\beta}R_{\mu\bar{\nu}\alpha\bar{\beta}}\sum_i h^\mu_i\xi_{i\alpha} \sum_j \overline{h^\nu_j\xi_{j\beta}},
}
where $h_i^\mu\coloneqq \frac{\p h^\mu}{\p z_i}$ is the matrix of $\rd h$. 
Hence, the right-hand side of \eqref{eq:34} is non-negative by Lemma~\ref{lemJ.12}  and  the claim follows.
\end{proof}

\begin{lemma}\label{lemJ.14}
We have
\eq{
    i\Theta(h^*T^{1,0}\CP^n\otimes\mathcal{O}(n))\geq 2n\,\rid_{h^*T^{1,0}\CP^n}\otimes\omega.
}
\end{lemma}
\begin{proof}
The curvature of the tensor product is
\eq{
    i\Theta(h^*T^{1,0}\CP^n\otimes\mathcal{O}(n)) = i\Theta(h^*T^{1,0}\CP^n)\otimes\rid_{\mathcal{O}(n)} + \rid_{h^*T^{1,0}\CP^n}\otimes i\Theta(\mathcal{O}(n)).
}
Note that
\eq{\label{eq:35}
    i\Theta(\mathcal{O}(1)) = 2\omega,
}
where $\omega$ is now the K\"ahler form of $\CP^n$. In fact, in the coordinates
\eq{
    [z_0:\cdots:z_n]=[1:w_1:\cdots:w_n],  \quad S=1+\abs{w}^2,
}
the holomorphic frame of $\mathcal{O}(-1)$ is
\eq{
    e = (1,w_1,\cdots,w_n), \quad \abs{e}^2 = S.
}
Hence the curvature is
\eq{
    i\Theta(\mathcal{O}(-1)) = -i\p\bar{\p}\log S, \quad i\Theta(\mathcal{O}(1)) = i\p\bar{\p}\log S.
}
The Fubini-Study metric gives the K\"ahler form
\eq{
    \omega = \frac{i}{2}\p\bar{\p}\log S,
}
hence we have \eqref{eq:35}. As a consequence,
\eq{
    i\Theta(\mathcal{O}(n)) = 2n\omega,
}
and hence Lemma~\ref{lemJ.13} yields
\eq{
    i\Theta(h^*T^{1,0}\CP^n\otimes\mathcal{O}(n)) \geq 2n\,\rid_{h^*T^{1,0}\CP^n}\otimes\omega,
}
the claim.
\end{proof}

\begin{lemma}\label{lemJ.15}
For any $\alpha\in\Omega^{n,1}(h^*T\CP^n\otimes\mathcal{O}(n))$, we have
\eq{
    \int_{\CP^n} \abs{\bar{\p}\alpha}^2 + \int_{\CP^n} \abs{\bar{\p}^*\alpha}^2 \geq 2n\int_{\CP^n} \abs{\alpha}^2.
}
\end{lemma}
\begin{proof}
Using the K\"ahler identity (see for instance \cite{Demailly07}*{Chapter 6, Corollary 5.9})
\eq{
    [L,\Lambda]\alpha = (p+q-n)\alpha, \quad \alpha\in\Omega^{p,q},
}
we have for our case $\alpha\in\Omega^{n,1}$ that
\eq{
    [\omega,\Lambda]\alpha = [L,\Lambda]\alpha = \alpha.
}
The conclusion then follows from \eqref{eq:20} and Lemma~\ref{lemJ.14}.
\end{proof}

In view of \eqref{eq:36}, we may isometrically identify each
\eq{
    \alpha\in\Omega^{n,1}(h^*T\CP^n\otimes\mathcal{O}(n))
}
with
\eq{
    U\in\Omega^{0,1}(h^*T\CP^n\otimes\mathcal{O}(-1)).
}
We now prove Proposition~\ref{propD.8}.

\begin{proof}[Proof of Proposition~\ref{propD.8}]
Recall the operator
\eq{
    \bar{\p}:\Gamma(\mathcal{E}(1)) \to \Omega^{0,1}(\mathcal{E}(1)).
}
Since
\eq{
    {\rm im}(\bar{\p}^*) = {\rm ker}(\bar{\p})^\perp,
}
the orthogonality condition
\eq{
    s\in\Gamma(\mathcal{E}(1)), \quad s\perp H^0(\mathcal{E}(1))={\rm ker}(\bar{\p})
}
implies that there exists $U\in\Omega^{0,1}(\mathcal{E}(1))$, such that
\eq{
    \bar{\p}^*U = s.
}
In view of the Hodge decomposition, we may further assume as a normalization that
\eq{
    \bar{\p}U = 0.
}
Inserting into Lemma~\ref{lemJ.15} yields (with the isometric identification $\alpha\mapsto U$)
\eq{\label{eq:38}
    \int_{\CP^n} \abs{s}^2 \geq 2n\int_{\CP^n} \abs{U}^2.
}
Moreover, we have
\eq{\label{eq:39}
    \int_{\CP^n} \abs{s}^2 = \int_{\CP^n} \<s,\bar{\p}^*U\> = \int_{\CP^n} \<\bar{\p}s,U\> \leq \Big( \int_{\CP^n} \abs{\bar{\p}s}^2 \Big)^{\frac{1}{2}} \Big( \int_{\CP^n} \abs{U}^2 \Big)^{\frac{1}{2}}.
}
Combining \eqref{eq:38} and \eqref{eq:39} gives
\eq{
    \int_{\CP^n} \abs{\bar{\p}s}^2 \geq 2n\int_{\CP^n} \abs{s}^2.
}
The claim follows.
\end{proof}

Finally, we deduce Theorem~\ref{thm5.1}.

\begin{proof}[Proof of Theorem~\ref{thm5.1}]
It is now a direct consequence of \eqref{eq:21.1}, \eqref{eq:23}, and Proposition~\ref{propD.8}.
\end{proof}

\section{Proof of Theorem~\ref{thm_nullity}: nullity part}\label{sec6}

In this section, we prove the following result, which completes the proof of Theorem~\ref{thm_nullity}.

\begin{theorem}\label{thmD.1}
For any degree-one holomorphic map $h:\CP^n\to\CP^n$, and $u=h\circ\pi$, we have
\eq{
    {\rm nul}(u)=3n^2+5n.
}
\end{theorem}

\begin{proof} When $m=1$, 
$Q_1$ is indefinite in $\Gamma (\mathcal{E}(1)) $. We have a  decomposition $\Gamma(\mathcal{E}(1))= H^0(\mathcal{E}(1))\oplus H^0(\mathcal{E}(1))^\perp$. In Subsection   \ref{subsection5.5}, we proved $Q_1= -(2n-1)\norm{\cdot}^2$ on $H^0$ and $Q_1 \geq (2n+1)\norm{\cdot}^2$ on $(H^0)^\perp$. Moreover, by \eqref{eq:10.1} and \eqref{eq:21.1} we see that the only possible cases for the nullity are $m=0$ and $m=2$.

\smallskip 
\noindent\textbf{Case $m=0$. } In this case, \eqref{eq:10.1} becomes
\eq{
    Q_0(s) = 2\int_{\CP^n}\abs{\bar{\p}s}^2.
}
Hence
\eq{
    Q_0(s)=0 \iff \bar{\p}s=0.
}
It is well known that the space of holomorphic sections of $T^{1,0}\CP^n$ has complex dimension $n^2+2n$, hence
\eq{
    \dim_\C H^0(h^*T^{1,0}\CP^n) = \dim_\C H^0(T^{1,0}\CP^n) = n^2+2n.
}
Therefore, the case $m=0$ provides nullity $2n^2+4n$.

\smallskip 

\noindent\textbf{Case $m=2$. } In this case, \eqref{eq:10.1} becomes
\eq{\label{eq:39.9}
    Q_2(s) = 2\int_{\CP^n} \Big( \abs{\bar{\p}s}^2 - (2n-2)\abs{s}^2 \Big).
}
Recall that
\eq{
    H_{\ell}=H_{\ell,A}=\{[z]\mid\ell(Az)=0\}.
}
Collecting the deficits in \eqref{eq:19.1} and \eqref{eq:21} gives
\eq{\label{eq:40}
    &\int_{\CP^n}\abs{\bar{\p}s}^2 - 2(n-1)\int_{\CP^n}\abs{s}^2 \\
    &= C(n)\int_{S(\C^{n+1})^*}\abs{A^*\ell}^{-2}\Big( \int_{H_{\ell,A}} \abs{N_{\ell,A}(\mathbf{b})}^2 \,\rdV_{H_{\ell,A}} + \int_{H_{\ell,A}}\abs{\p^{H_{\ell,A}}\ell(\mathbf{b})}^2 \,\rdV_{H_{\ell,A}} \Big)\rd\sigma(\ell),
}
where $N_{\ell,A}(\mathbf{b})$ is the normal component of $\ell(\bar{\p}\mathbf{b})$ along $H_{\ell,A}$ (i.e. the deficit in \eqref{eq:42}).
Recall that $s\in\Gamma(h^*T^{1,0}\CP^n\otimes\mathcal{O}(-2))$ is the image of $\mathbf{b}\in\Gamma(\underline{\C}^{n+1}\otimes\mathcal{O}(-1))$ in \eqref{eq:12}. Set
\eq{
    \mathbf{b}_0\coloneqq A^{-1}\mathbf{b}, \qquad \Phi_A:S(\C^{n+1})^*\to S(\C^{n+1})^*, \quad \Phi_A(\ell) = \frac{A^*\ell}{\abs{A^*\ell}} = \ell_0.
}
Then $H_{\ell,A}=H_{\Phi_A(\ell),I}$. Hence
\eq{
    \ell(\mathbf{b}) = (A^*\ell)(\mathbf{b}_0) = \abs{A^*\ell}\,\Phi_A(\ell)(\mathbf{b}_0),
}
and
\eq{\label{eq:41}
    \ell(\bar{\p}\mathbf{b}) = \abs{A^*\ell}\,\Phi_A(\ell)(\bar{\p}\mathbf{b}_0), \qquad \p^{H_{\ell,A}}\ell(\mathbf{b}) = \abs{A^*\ell}\,\p^{H_{\Phi_A(\ell),I}}\Phi_A(\ell)(\mathbf{b}_0).
}
Using
\eq{
    \abs{\ell(\bar{\p}\mathbf{b})}^2 = \abs{\bar{\p}^{H_{\ell,A}}\ell(\mathbf{b})}^2 + \abs{N_{\ell,A}(\mathbf{b})}^2,
}
we have
\eq{\label{eq:41.5}
    \abs{N_{\ell,A}(\mathbf{b})}^2 = \abs{A^*\ell}^2 \abs{N_{\Phi_A(\ell),I}(\mathbf{b}_0)}^2.
}
For simplicity, we denote
\eq{\label{eq:43}
    F(\ell) \coloneqq \int_{H_{\ell,I}} \abs{N_{\ell,I}(\mathbf{b}_0)}^2 \,\rdV_{H_{\ell,I}} + \int_{H_{\ell,I}}\abs{\p^{H_{\ell,I}}\ell(\mathbf{b}_0)}^2 \,\rdV_{H_{\ell,I}}.
}
Combining \eqref{eq:40}, \eqref{eq:41}, \eqref{eq:41.5}, and \eqref{eq:43} gives
\eq{
    \int_{\CP^n}\abs{\bar{\p}s}^2 - 2(n-1)\int_{\CP^n}\abs{s}^2 = C(n)\int_{S(\C^{n+1})^*} F(\Phi_A(\ell)) \,\rd\sigma(\ell) = C(n)\int_{S(\C^{n+1})^*} F(\ell) \,\rd\sigma_A(\ell),
}
where $\rd\sigma_A=(\Phi_A)_*(\rd\sigma)$. Since $\Phi_A:S(\C^{n+1})^*\to S(\C^{n+1})^*$ is a diffeomorphism, we see that $\rd\sigma$ and $\rd\sigma_A$ are equivalent measures. Hence by \eqref{eq:39.9} we have
\eq{
    Q_2(s)=0 \iff \int_{S(\C^{n+1})^*} F(\ell) \,\rd\sigma_A(\ell)=0 \iff \int_{S(\C^{n+1})^*} F(\ell) \,\rd\sigma(\ell)=0.
}
Note that the last statement is independent of $A$. Let $s_0\in\Gamma(T^{1,0}\CP^n\otimes\mathcal{O}(-2))$ be the image of $\mathbf{b}_0$. Choosing $A=I$ in the above discussion, or directly using \eqref{eq:20}, we have
\eq{
    C(n)\int_{S(\C^{n+1})^*} F(\ell) \,\rd\sigma(\ell) = \int_{\CP^n}\abs{\bar{\p}s_0}^2 - 2(n-1)\int_{\CP^n}\abs{s_0}^2 = \int_{\CP^n} \abs{\p s_0}^2.
}
Hence $Q_2(s)=0$ if and only if $s_0$ is an anti-holomorphic section of $T^{1,0}\CP^n\otimes\mathcal{O}(-2)$. Since the duality exchanges $(1,0)$ type and $(0,1)$ type, namely
\eq{
    \bar{\p}(s_0^\flat) = (\p s_0)^\flat,
}
we have
\eq{
    \{s_0\mid\p s_0=0\} \cong H^0((T^{1,0}\CP^n\otimes\mathcal{O}(-2))^*) = H^0(\Omega^{1,0}\CP^n\otimes\mathcal{O}(2)).
}
Choosing $m=2$ in \eqref{eq:11} and taking duality gives
\eq{
    0 \to \Omega^{1,0}\CP^n\otimes\mathcal{O}(2) \to (\underline{\C}^{n+1})^*\otimes\mathcal{O}(1) \to \mathcal{O}(2) \to 0.
}
Note that
\eq{
    H^0(\mathcal{O}(1)) = (\underline{\C}^{n+1})^*, \qquad H^0(\mathcal{O}(2)) = {\rm Sym}^2(\underline{\C}^{n+1})^*.
}
Since
\eq{
    (\underline{\C}^{n+1})^*\otimes(\underline{\C}^{n+1})^* = {\rm Sym}^2(\underline{\C}^{n+1})^* \oplus \Lambda^2(\underline{\C}^{n+1})^*,
}
we see that 
\eq{
    H^0(\Omega^{1,0}\CP^n\otimes\mathcal{O}(2)) \cong \Lambda^2(\underline{\C}^{n+1})^*,
}
that is, the space of anti-symmetric complex $(n+1)\times(n+1)$ matrices, which has complex dimension
\eq{
    \dim_\C\{C\in M_{n+1}(\C) \mid C^T+C=0 \} = \frac{1}{2}(n^2+n).
}
Hence the case $m=2$ provides nullity $n^2+n$.

Altogether we have
\eq{
    {\rm nul}(h\circ\pi) = 2n^2+4n + n^2+n = 3n^2+5n.
}

\end{proof}

\appendix

\section{The Sasakian structure and the Hopf map}\label{appendix_A}

In this Appendix, for the convenience of the reader, we recall the Sasakian structure on $\S^{2n+1}$, the Hopf map, and holomorphic maps from $\mathbb{CP}^n$ into itself. 

We recall that for an odd-dimensional manifold $(M,g)$, a \textit{Sasakian structure} $(\xi,\Phi)$, consisting of a vector field $\xi$ and a $(1,1)$-tensor field $\Phi$, is defined by
\eq{
    \Phi^2(X) = -X+g(X,\xi)\xi,\quad g(\Phi(X),\Phi(Y)) = g(X,Y)-g(X,\xi)g(Y,\xi), \quad\forall X,Y\in\Gamma(TM)
}
and
\eq{\label{def_Sasakian}
    (\nabla_X\Phi)(Y) = -g(X,Y)\xi + g(Y,\xi)X, \quad\forall X,Y\in\Gamma(TM).
}
Here $\xi$ is called the \textit{Reeb field}, and $\Phi=\nabla\xi$.

In particular, the standard sphere $\S^{2n+1}\subset\mathbb{C}^{n+1}$ has a standard Sasakian structure $(\xi,\Phi)$ defined by
\eq{
    \xi \coloneqq \mathbf{j}x, \quad \Phi(X) \coloneqq \mathbf{j}X + \<X,\xi\>x, \quad\forall X\in\Gamma(T\S^{2n+1}),
}
where $\mathbf{j}$ is the standard complex structure on $\mathbb{C}^{n+1}$. In fact, $\Phi$ is just the projection of $\mathbf{j}$ onto the tangent space of $\S^{2n+1}$, since
\eq{
    \Phi(X) = \mathbf{j}X + \<X,\mathbf{j}x\>x = \mathbf{j}X - \<\,\mathbf{j}X,x\>x.
}

We view the sphere as
\eq{
    \S^{2n+1} = \{z=(z_0,z_1,\cdots,z_n)\in\mathbb{C}^{n+1}\,:\, \abs{z}=1\}.
}
The complex projective space $\mathbb{CP}^n$ is endowed with the Fubini-Study metric
\eq{
    g_{i\bar{j}} = \frac{(1+\abs{z}^2)\delta_{i\bar{j}} - \bar{z}_iz_j}{(1+\abs{z}^2)^2}.
}
The standard \textit{Hopf map} (or \textit{Hopf fibration}) $\pi:\S^{2n+1}\to\mathbb{CP}^n$ is defined by
\eq{
    \pi(z) \coloneqq [z_0:z_1:\cdots:z_n].
}
It is a smooth Riemannian submersion, and is a harmonic map.

By definition, the Hopf map $\pi$ projects every complex line $\mathbb{C}z$ to $[z]$. Therefore $\pi$ can be constructed as the quotient map of a principal bundle
\eq{
    \S^1 \rightarrow \S^{2n+1} \xrightarrow{\pi} \mathbb{CP}^n,
}
where $\S^1$ freely acts on $\S^{2n+1}$ by
\eq{
    e^{i\theta}\cdot z \coloneqq e^{i\theta}z.
}
Each $\S^1$-fiber of the Hopf map $\pi$ is the orbit of the flow induced by the Reeb field $\xi$. The horizontal subbundle is defined by $\xi^\perp\subset T\S^{2n+1}$. The Hopf map $\pi$ is horizontally holomorphic (or transversely holomorphic) in the sense that
\eq{
    J \circ \rd\pi = \rd\pi \circ \Phi,
}
where $J$ is the complex structure on $\mathbb{CP}^n$.

A \textit{holomorphic map} between complex manifolds is a map $h:(M, J^M)\to (N, J^N)$ that satisfies
\eq{
    J^N \circ \rd h = \rd h \circ J^M.
}
All holomorphic maps from $\mathbb{CP}^n$ into itself are characterized by
\eq{
    h([z_0:z_1:\cdots:z_n]) = [p_0(z):p_1(z):\cdots:p_n(z)],
}
where $p_0,p_1,\cdots,p_n$ are homogeneous polynomials with the same degree, and have no common zeros except $0$. For example, for any $d\in\mathbb{N}$,
\eq{
    h([z_0:z_1:\cdots:z_n]) = [z_0^d:z_1^d:\cdots:z_n^d]
}
is a holomorphic map from $\mathbb{CP}^n$ into itself.

\section{An example}\label{appendix_B}

We give an example to show that the composition of the Hopf map $\pi:\S^{2n+1}\to\mathbb{CP}^n$ and a holomorphic map from $\mathbb{CP}^n$ to itself is not necessarily horizontally weakly conformal, i.e.
\eq{
    (\rd u)(\rd u)^* = \lambda(x)^2 \,\rid
}
is not necessarily true when $n\geq2$. 

We consider $n=2$. Let $\pi:\S^5\to\mathbb{CP}^2$ be the Hopf map and $h:\mathbb{CP}^2\to\mathbb{CP}^2$ be a holomorphic map with algebraic degree 2 defined by
\eq{
    h([Z_0,Z_1,Z_2]) \coloneqq [Z_0^2,Z_1^2,Z_2^2],
}
and set
\eq{
    u \coloneqq h \circ \pi : \S^5\to\mathbb{CP}^2.
}
In the chart $\{Z_0\not=0\}$, we set
\eq{
    z_1 \coloneqq \frac{Z_1}{Z_0}, \quad z_2 \coloneqq \frac{Z_2}{Z_0},
}
then $h$ has a simple expression
\eq{
    h(z_1,z_2) = (z_1^2,z_2^2),
}
hence
\eq{
    \rd h|_{(z_1,z_2)} = \begin{pmatrix}
        2z_1 & 0 \\
        0 & 2z_2
    \end{pmatrix}.
}
Choose a small enough $\epsilon>0$ and consider
\eq{
    p \coloneqq [1:\epsilon:2\epsilon] = (\epsilon, 2\epsilon) \in\mathbb{CP}^2.
}
It follows
\eq{
    \rd h|_p(\p_{z_1}) = 2\epsilon\, \p_{z_1}, \quad \rd h|_p(\p_{z_2}) = 4\epsilon\, \p_{z_2}.
}
Since the Hopf map $\pi$ is horizontally surjective, there exists $x\in\S^5$ such that $\pi(x)=p$ and horizontal vectors $X_1,X_2\in T_x\S^5$ such that
\eq{
    \rd\pi|_x(X_1) = \frac{\p_{z_1}}{\abs{\p_{z_1}}}, \quad \rd\pi|_x(X_2) = \frac{\p_{z_2}}{\abs{\p_{z_2}}}.
}
Note that the Fubini-Study metric
\eq{
    g_{i\bar{j}} = \frac{(1+\abs{z}^2)\delta_{i\bar{j}} - \bar{z}_iz_j}{(1+\abs{z}^2)^2}
}
has the approximation
\eq{
    g(z) = \delta + O(\abs{z}^2) \qquad\hbox{as}\quad z\to0.
}
Hence
\eq{
    \abs{\rd u_x(X_1)} = \abs{\rd h_p\Big(\frac{\p_{z_1}}{\abs{\p_{z_1}}}\Big)} = 2\epsilon + O(\epsilon^3),
}
and similarly
\eq{
    \abs{\rd u_x(X_2)} = \abs{\rd h_p\Big(\frac{\p_{z_2}}{\abs{\p_{z_2}}}\Big)} = 4\epsilon + O(\epsilon^3).
}
Therefore $u$ is not horizontally weakly conformal. It is clear that the same reason holds for general $n\geq2$ that $h \circ \pi :\S^{2n+1}\to\mathbb{CP}^n$ is not necessarily horizontally weakly conformal.

\bigskip
\noindent{\sc Acknowledgment.}
The main part of the work, especially Theorem \ref{Theorem_R}, was reported by G. W. in a workshop in Wuhan in April, 2026. 
M. Z. would like to thank Linlin Sun for stimulating discussion about Theorem \ref{thm_nullity} there.

\bibliographystyle{alpha}
\bibliography{BibTemplate.bib}

\end{document}